\documentclass[11pt]{amsart}
\usepackage{amsmath}
\usepackage{amssymb,amsthm}
\usepackage{latexsym}
\usepackage{amsfonts}
\usepackage{xcolor}
\usepackage[plainpages=false, hypertexnames=false, pdfpagelabels=true, hyperindex=true, linktocpage, pagebackref=true, pdfa=true]{hyperref}
\usepackage{amsrefs}
\usepackage{mathrsfs}
\usepackage{multirow}

\begin{document}
\newcommand{\ci}[1]{_{ {}_{\scriptstyle #1}}}
\newcommand{\norm}[1]{\ensuremath{\|#1\|}}
\newcommand{\abs}[1]{\ensuremath{\vert#1\vert}}
\newcommand{\p}{\ensuremath{\partial}}
\newcommand{\pbar}{\ensuremath{\bar{\partial}}}
\newcommand{\db}{\overline\partial}
\newcommand{\D}{\mathbb{D}}
\newcommand{\DD}{\mathbb{D}}
\newcommand{\T}{\mathbb{T}}
\newcommand{\C}{\mathbb{C}}
\newcommand{\CC}{\mathbb{C}}

\newcommand{\N}{\mathbb{N}}
\newcommand{\td}{\widetilde\Delta}
\newcommand{\La}{\langle }
\newcommand{\Ra}{\rangle }
\newcommand{\tr}{\operatorname{tr}}
\newcommand{\ran}{\operatorname{Ran}}
\newcommand{\vf}{\varphi}
\newcommand{\e}{\varepsilon}
\newcommand{\f}[2]{\ensuremath{\frac{#1}{#2}}}
\newcommand{\be}{\mathbf{e}}
\newcommand{\clos}{\operatorname{clos}}
\newcommand{\rank}{\operatorname{rank}}
\newcommand{\bz}{\mathbf{z}}
\newcommand{\tto}{\!\!\to\!}
\newcommand{\wt}{\widetilde}
\newcommand{\shto}{\raisebox{.3ex}{$\scriptscriptstyle\rightarrow$}\!}
\newcommand{\bL}{\mathbf{L}}
\newcommand{\entrylabel}[1]{\mbox{#1}\hfill}




\numberwithin{equation}{section}
\newtheorem{thm}{Theorem}[section]
\newtheorem{lemma}[thm]{Lemma}
\newtheorem{cor}[thm]{Corollary}
\newtheorem{definition}[thm]{Definition}

\newtheorem{prop}[thm]{Proposition}
\newtheorem{rem}[thm]{Remark}
\newtheorem*{rem*}{Remark}

\title[The Matrix Corona Problem and Slice Holomorphic Functions]{The Corona Problem for Slice Holomorphic Functions via the Matrix Corona Problem}

\author[F. Colombo]{Fabrizio Colombo}
\address{(FC)
Politecnico di Milano\\Dipartimento di Matematica\\Via E. Bonardi 9\\20133
Milano, Italy}
\email{fabrizio.colombo@polimi.it}

\author[E. Pozzi]{Elodie Pozzi}
\address{(EP) Department of Mathematics and Statistics \\
Saint Louis University \\
220 N. Grand Blvd, 63103 St Louis MO, USA}
\email{elodie.pozzi@slu.edu}

\author[I. Sabadini]{Irene Sabadini}
\address{(IS)
Politecnico di Milano\\Dipartimento di Matematica\\Via E. Bonardi 9\\20133
Milano, Italy
}
\email{irene.sabadini@polimi.it}

\author[B. D Wick]{Brett D. Wick}
\address{(BDW)  Department of Mathematics, Washington University -
 St. Louis, One Brookings Drive\\ St. Louis, MO USA 63130-4899}
\email{wick@math.wustl.edu}
\thanks{BDW's research supported in part by National Science Foundation DMS award \#2349868  and Australian Research Council -- DP260201083.
}

\thanks{EP and BDW would like to thank Politecnico di Milano for hospitality during visits in Spring 2024 when this project was started.}

\begin{abstract}
In this paper we prove the Corona theorem for slice hyperholomorphic functions in a general way that applies to functions with values in a Clifford algebra $\mathbb R_n$ for any $n\geq 2$. This general proof is based on a matrix version of Corona's theorem and then exploiting some symmetries of the matrices that allow to reconstruct the values of the functions. The crucial idea is to tie together the algebra of the functions, the B\'ezout identities and the Carleson-type conditions. The results we obtain rest on a new way of writing the slice product which improves the standard approach and which holds in general dimension.
\end{abstract}

\maketitle



\section{Introduction}

Carleson's celebrated Corona Theorem \cite{MR0141789} states that for $f_1,\ldots,f_N\in H^\infty(\D)$ with
\[
0<\delta^2\le \sum_{j=1}^N |f_j(z)|^2\le 1,\qquad z\in\D,
\]
there exist $h_1,\ldots,h_N\in H^\infty(\D)$ satisfying the B\'ezout identity
\[
\sum_{j=1}^N f_j(z)h_j(z)=1,\qquad z\in\D,
\]
with norm estimates $\max_j \norm{h_j}_{H^\infty(\D)}\lesssim C(\delta,N)$.  Beyond its intrinsic appeal, the result is a cornerstone of function theory, operator theory, and harmonic analysis, and it has motivated a substantial body of work on variants and generalizations; see the survey \cite{MR3329539} and the monographs \cites{MR1419088,MR2261424,MR1669574,MR0827223} for accounts.

In several complex variables the picture is markedly different: Sibony \cite{MR0916722} produced domains in $\C^n$ on which the $H^\infty$-Corona problem fails, while the question remains famously open on the unit ball and the polydisc.  A different multidimensional framework arises from the Fueter--Sce--Qian construction \cite{MR4240465}, whose modern starting point is the set of hyperholomorphic functions, which are also called slice regular when they are functions of a quaternionic variable (and quaternionic valued) \cites{MR3013643,MR2752913,MR3585395} and are called slice monogenic when they are functions of a paravector variable and taking values in a Clifford algebra \cites{MR2089988,MR3496962,MR1169463}.  Both theories support rich function-theoretic and spectral-theoretic developments \cites{MR3967697,MR3887616,MR3930595,MR2069293}, and their interaction has been further developed through the fine structures on the $S$-spectrum \cites{MR4502763,MR4576314,MR4674273,MR4741057}.

A characteristic difficulty in the hyperholomorphic setting is that pointwise multiplication does not preserve hyperholomorphicity.  One therefore replaces the pointwise product with the \emph{slice product} (or $\ast$-product) introduced in \cites{MR2684426,MR2520116}, which turns the relevant spaces into associative (but, in general, non-commutative) algebras of hyperholomorphic functions.  In light of this, the natural Corona-type problem reads: given slice hyperholomorphic functions $f_1,\ldots,f_N$ on the unit ball $\mathbb B$ of $\mathbb H$ (or, more generally, slice monogenic functions valued in $\mathbb{R}_n$ i.e., $f_j:\, \mathbb{B}\subset\mathbb{R}^{n+1}\to\mathbb{R}_n$,) with
\[
0<\delta^2\le \sum_{j=1}^N |f_j(q)|^2\le 1,
\]
find slice hyperholomorphic $h_1,\ldots,h_N$ satisfying
\[
\sum_{j=1}^N (f_j\ast h_j)(q)=1\quad\textnormal{ and }\quad \max_{1\leq n\leq N}\|h_j\|_{H^\infty(\mathbb{B})}\leq C(\delta,N).
\]
An early, qualitative result in this direction was obtained by Gentili--Sarfatti--Struppa in \cite{MR3621101} via sheaf-theoretic methods, where ideals of slice regular functions were studied but quantitative bounds were not produced.  A related result under a different non-vanishing hypothesis appears in \cite{Shelah_MSThesis}, and the question was explicitly raised in the survey \cite{Saracco}.  The quaternionic Corona theorem was recently proved in quantitative form by the present authors in \cite{CPSW}, where the case of $\mathbb{R}_n$-valued slice monogenic functions was raised as an open question.

The proof in \cite{CPSW} proceeds by adapting Wolff's approach to Carleson's theorem (see \cites{MR2261424,MR1419088,MR1669574,MR0827223}) to the hyperholomorphic setting.  By the Splitting Lemma, Lemma \ref{lem:splitting} below, write each $f_j=F_j+G_jJ$ on a fixed complex slice $\mathbb{C}_I\subset \mathbb H$ for an imaginary unit $J\perp I$, where $F_j,G_j\in H^\infty(\D)$, the slice B\'ezout equation decomposes into a coupled pair of B\'ezout-type equations on the disk,
\begin{align*}
\sum_{j=1}^{N}\bigl(F_jH_j-G_j\hat K_j\bigr)&=1,\\
\sum_{j=1}^{N}\bigl(F_jK_j+G_j\hat H_j\bigr)&=0,
\end{align*}
where $\hat F(z):=\overline{F(\bar z)}$ is the natural involution on $H^\infty(\D)$.  The Wolff-type proof then constructs smooth solutions, perturbs them to be holomorphic via a system of $\bar\partial$-equations, and yields bounded holomorphic solutions.  A central technical point in \cite{CPSW} is that the simultaneous solvability of the two displayed equations forces a \emph{second}, less obvious Carleson-type condition: in the case $N=1$,
\begin{equation}\label{eq:intro-aux}
(F\hat F+G\hat G)(H\hat H+K\hat K)=1,
\end{equation}
and for general $N$ a more complicated quadratic identity (recalled in Section~\ref{s:simpleproof}).  In \cite{CPSW} these identities are derived by direct algebraic manipulation, and the corresponding Carleson conditions are imposed by hand to ensure that the relevant block matrices are bounded below.

\medskip

The present paper has three aims, all centered on a single observation that ties together the algebra of slice (and slice monogenic) functions, the auxiliary B\'ezout identities, and the Carleson-type conditions.

\medskip
\noindent\textbf{(1) The slice product is a matrix product.}  The restriction of a slice monogenic function $f_I$ on a fixed slice $\mathbb{C}_I$ corresponds to a $2^{n-1}\times 2^{n-1}$ matrix $M_{f_I}(z)$ with entries in $H^\infty(\D)$, built from the components of $f_I$ in the Splitting Lemma decomposition together with the involution $\hat\cdot$.  In the quaternionic case ($n=2$) this is the familiar Cayley-type representation
\[
M_f(z)=\begin{pmatrix}F(z)&-G(z)\\ \hat G(z)&\hat F(z)\end{pmatrix},\qquad z\in\mathbb{B}\cap \mathbb{C}_I\cong\mathbb{D},
\]
and in general we give an explicit closed-form description of $M_f$ in terms of symmetric differences and Clifford signs that arise from the $\ast$-product (Section~\ref{subsec:matrix-clifford-aux}).  The key structural property is the homomorphism
\[
M_{f\ast g}(z)=M_f(z)\,M_g(z),
\]
so that the multi-generator B\'ezout identity $\sum_j f_j\ast h_j=1$ is equivalent to the matrix identity
\[
\sum_{j=1}^N M_{f_j}(z)\,M_{h_j}(z)=\mathrm{Id}_{2^{n-1}}.
\]
We give a self-contained derivation of $M_f$ and of the general explicit formula for $f_I\ast h_I$ as a sum over symmetric differences, in the process simplifying the description in \cite{MR2684426}.

\medskip
\noindent\textbf{(2) A simple matrix-Corona proof of the slice hyperholomorphic and slice monogenic Corona theorems.}  Once the slice B\'ezout equation has been reformulated as a matrix identity, the natural strategy is to factor the matrix sum as a product of a wide block matrix $\mathsf{F}=(M_{f_1}\cdots M_{f_N})$ and a tall block matrix and to invoke the classical \emph{matrix Corona theorem} for $H^\infty(\D)$ \cites{MR0383098,MR0990193,MR0629839,MR0595742,MR2362422,MR2105955,MR2106344,MR2158178,MR2213732,MR2317961,MR1887635}.  Doing so produces some bounded matrix right inverse $\mathsf{P}$ of $\mathsf{F}$, with quantitative norm bounds inherited from the matrix Corona theorem.  This right inverse will in general not have the structured block form required to recover a slice solution: its entries are arbitrary $H^\infty$ functions, not constrained to be conjugate-pairs in the way that the $M_{h_j}$ are.  The key observation is that the structured block form is precisely the joint fixed subspace of a family of explicit commuting involutions $\mathcal T_2,\ldots,\mathcal T_n$ on the space of right inverses, induced by left multiplication by the inverse of the Clifford units $I_2,\ldots,I_n$ together with the action of $\hat\cdot$.  Averaging $\mathsf{P}$ over the abelian group $\langle \mathcal T_2,\ldots,\mathcal T_n\rangle$ produces a new bounded right inverse $\mathsf{H}$ that automatically has the block structure of a slice solution, and the components of $\mathsf{H}$ are the desired $h_j$'s.  In the quaternionic case, this projection is the single-involution symmetrization $\tfrac{1}{2}(\mathsf{P}+J_{2N}^{-1}\hat{\mathsf{P}} J_2)$, and in general it is the $2^{n-1}$-term average
\[
\mathsf{H}=\frac{1}{2^{n-1}}\sum_{S\subseteq\{2,\ldots,n\}}\mathcal T_S(\mathsf{P}).
\]
This yields a markedly simpler proof of the quaternionic Corona theorem of \cite{CPSW} (Section~\ref{s:simpleproof}), and it carries over with no essential change to the slice monogenic setting over $\mathbb{R}_n$ for arbitrary $n\ge 2$, providing a positive answer to the question raised in \cite{CPSW}.  Since the averaging procedure does not increase norms, the quantitative bounds from the matrix Corona theorem transfer directly to the slice solutions.  Our main result is the following.

\begin{thm}[Slice monogenic Corona theorem]\label{thm:main-intro}
Let $n\ge 2$, $N\in\mathbb{N}\cup\{\infty\}$, and $0<\delta<1$.  Let $\mathbb B\subseteq\mathbb{R}^{n+1}$ be the unit ball and let $f_{1},\ldots,f_{N}\in H^\infty(\mathbb B)$ be slice monogenic functions valued in $\mathbb{R}_n$.  Suppose the data satisfy the Carleson-type condition
\begin{equation}\label{Carltype cond}
0<\delta^2\le \sum_{j=1}^N |f_{j}(q)|^2\le 1\quad\text{for all }q\in \mathbb B,
\end{equation}
together with the second, quadratic Carleson condition arising from $\det(\mathsf{F}\,\mathsf{F}^\ast)\ge\delta^2$, for all $z\in\mathbb B\cap\mathbb C_I$, for the block matrix $\mathsf{F}=(M_{f_1}\cdots M_{f_N})$ (made explicit in Theorem~\ref{thm:master-formula}).  Then there exist slice monogenic functions $h_{1},\ldots,h_{N}\in H^\infty(\mathbb B)$ satisfying the B\'ezout identity
\[
\sum_{j=1}^{N} (f_{j}\ast h_{j})(q)=1\quad\text{for all }q\in \mathbb B,
\]
with norm estimates
\[
\max_{1\le j\le N}\norm{h_{j}}_{H^\infty(\mathbb B)}\lesssim C(\delta,n),
\]
 where $C(\delta,n)$ is the constant in the matrix-Corona estimate for the $2^{n-1}\times N\cdot 2^{n-1}$ block matrix $\mathsf{F}$.  In particular, $C(\delta,n)$ depends polynomially on $1/\delta$ with exponent and coefficients depending on $n$ and $N$; for $n=2$ it specializes to the explicit constant $C(\delta)$ of \cite{CPSW}.
\end{thm}

\medskip
\noindent\textbf{(3) The auxiliary B\'ezout identities are determinant identities.}  The identity \eqref{eq:intro-aux} and its higher-$N$ analogues, derived in \cite{CPSW} by algebraic manipulation, become conceptually transparent in the matrix formulation.  Taking determinants of $M_fM_g=\mathrm{Id}$ gives, in the quaternionic $N=1$ case,
\[
\det(M_f)\det(M_g)=1,\qquad \det(M_f)={F}\hat{{F}}+G\hat{G},
\]
which is exactly \eqref{eq:intro-aux}.  For general $N$, the auxiliary identity arises from $\det\!\bigl(\sum_j M_{f_j}M_{h_j}\bigr)=1$; expanding this determinant via the Cauchy--Binet formula and a generalized Laplace expansion produces a sum of squared moduli that is precisely the Carleson-type quantity appearing in \cite{CPSW}.  This perspective makes manifest the geometric origin of the auxiliary condition: it is the Gram-matrix determinant of the block matrix $\mathsf{F}$, and it is automatically a sum of squares (hence nonnegative) by Cauchy--Binet, with Sylvester's criterion certifying that $\mathsf{F}\,\mathsf{F}^\ast$ is positive semidefinite (Corollary~\ref{cor:psd-sylvester}).  In the slice monogenic case the same computation produces a more intricate sum-of-squares identity; we develop this Master Formula in the Appendix (Theorem~\ref{thm:master-formula}).

\medskip

While this article was being finalized, we learned of the paper of Lin, Lu, and Zu~\cite{LinLuZu2026}, posted to the arXiv on August 9, 2026. That paper uses that the quaternionic slice product can be reformulated as a $2 \times 2$ matrix product, see for example \cite{MR3554256}, and combines this observation with the operator-valued Corona theorem of Treil--Wick~\cite{MR2158178} to give a new proof of the finite-generator quaternionic Corona theorem of~\cite{CPSW}; it further extends the result to countably many generators and establishes a Toeplitz corona theorem on the quaternionic Hardy space, thereby answering the two questions raised in~\cite{CPSW}*{Section~4}. The matrix-realization idea underlying our Section~2 was utilized independently, as part of the present project, which began in Spring 2024 and was directed principally at the Clifford-algebra generalization described below. The overlap between the two papers is confined to the quaternionic matrix-realization idea and application of the matrix Corona theorem:
 \cite{LinLuZu2026} treats countably many generators and the $H^2$/Toeplitz formulation in the quaternionic case, while the present paper extends the matrix-realization strategy to slice monogenic functions valued in $\mathbb{R}_n$ for arbitrary $n$, and identifies the associated auxiliary Carleson-type conditions as Cauchy--Binet determinant identities (Section~7).

\medskip

The paper is organized as follows.  Section~\ref{s:simpleproof} gives the matrix-Corona proof of the slice hyperholomorphic Corona theorem of \cite{CPSW} in the quaternionic case, first for $N=2$ and then in general, including the blockwise symmetry computation that organizes the symmetrization.  Section~\ref{sec:slice-monogenic-preliminaries} collects the preliminaries on slice monogenic functions following \cites{MR2684426,MR2520116} and an improved formula for the $*$-product.  In Section~\ref{subsec:matrix-clifford-aux} we develop the closed-form expression for the slice product and derive the matrix representation $M_f$ in the general Clifford setting. In Section \ref{Sect5} we develop the explicit calculations in the case of the Clifford algebra $\mathbb R_3$. Section~\ref{subsec:Bezout-general} carries out the symmetrization argument for $\mathbb{R}_n$ for arbitrary $n\ge 2$ and establishes the slice monogenic Corona theorem (Theorem~\ref{thm:corona-general}).  Section~\ref{s:LinearAlgebra} (the Appendix) develops the determinant identities via Cauchy--Binet that make the auxiliary B\'ezout equations transparent, including the Master Formula and its consequences for principal minors and positive semidefiniteness.

\section{Alternative Proof via Matrix Corona Theorem}
\label{s:simpleproof}

In this section we give an alternative, simpler proof of the recent result of the authors in \cite{CPSW} of the Corona problem for slice hyperholomorphic functions and which can be used in higher dimensions.  The idea is to rephrase the system of equations as a matrix equation, then apply the matrix Corona Theorem to that problem.  This produces a solution, but in general it lacks the necessary invariances and symmetries needed in the original equations.  However, a simple symmetrization of this matrix solution produces functions that solve the original system.  We carry this out now, first in the case of $N=2$ since it explains the computations and ideas clearly, then we turn to the case of general $N$.

The general structure used throughout this section, and continued in the Clifford setting in later sections, is the following.  The slice product $f*g$ can be represented by a matrix product $M_f M_g$, where $M_f$ is an appropriately sized matrix built from the components of $f$ and the involution
\[
\hat F(z) := \overline{F(\bar z)}.
\]
Consequently the B\'ezout-type equation $\displaystyle\sum_{j=1}^{N} f_j * h_j = 1$ becomes the matrix identity
\[
\sum_{j=1}^N M_{f_j}(z) M_{h_j}(z) = \textnormal{Id}_2,
\]
which factors as the product of a wide block matrix on the left and a tall block matrix on the right:
\[
\bigl(M_{f_1}\ \cdots\ M_{f_N}\bigr) \begin{pmatrix} M_{h_1} \\ \vdots \\ M_{h_N} \end{pmatrix} = \textnormal{Id}_2.
\]
The matrix Corona Theorem applied to the wide block matrix produces a tall right inverse, which we then symmetrize to recover the desired block structure.

\subsection{Two Generators}

Suppose that we are given $F_1,F_2, G_1,G_2\in H^\infty(\mathbb{D})$ and $H_1,H_2, K_1,K_2\in H^\infty(\mathbb{D})$ so that
\begin{align}
    F_1H_1+F_2H_2-G_1\hat{K}_1-G_2\hat{K}_2 &=1\label{e:e1}\\
    F_1K_1+F_2K_2+G_1\hat{H}_1+G_2\hat{H}_2 &= 0\label{e:e2}.
\end{align}
For the given $F_1,F_2, G_1,G_2\in H^\infty(\mathbb{D})$ define matrices
$$
\mathsf{F}(z)=
\left(
\begin{array}{cccc}
F_1(z) & -G_1(z) & F_2(z) & -G_2(z) \\
\hat G_1(z) & \hat F_1(z) & \hat G_2(z) & \hat F_2(z)
\end{array}
\right)
$$
and
$$
\mathsf{H}(z)=
\left(
\begin{array}{cc}
H_1(z) & -K_1(z) \\
\hat K_1(z) & \hat H_1(z) \\
H_2(z) & -K_2(z) \\
\hat K_2(z) & \hat H_2(z)
\end{array}
\right).
$$
By writing the restrictions of $f_i$ and $h_i$ as $f_i(z)=F_i(z)+G_i(z)J$, $h_i(z)=H_i(z)+K_i(z)J$, $i=1,2$,
the matrix $\mathsf{F}(z)$ is the wide $2\times 4$ matrix formed by placing the $2\times 2$ slice matrices $M_{f_1}$ and $M_{f_2}$ side by side, and $\mathsf{H}(z)$ is the tall $4\times 2$ matrix formed by stacking $M_{h_1}$ and $M_{h_2}$ vertically.  Observe that by direct matrix multiplication and the use of \eqref{e:e1} and \eqref{e:e2} directly, taking the $\widehat{\cdot}$ of the equation, or multiplying by $-1$, we have, omitting the variable $z$ in the computations,
\begin{align*}
\mathsf{F}\mathsf{H} & =\left(
\begin{array}{cccc}
F_1 & -G_1 & F_2 & -G_2 \\
\hat G_1 & \hat F_1 & \hat G_2 & \hat F_2
\end{array}
\right)\left(
\begin{array}{cc}
H_1 & -K_1 \\
\hat K_1 & \hat H_1 \\
H_2 & -K_2 \\
\hat K_2 & \hat H_2
\end{array}
\right)\\
&=\left(
\begin{array}{cc}
F_1 H_1 - G_1 \hat K_1 + F_2 H_2 - G_2 \hat K_2 & -F_1 K_1 - G_1 \hat H_1 - F_2 K_2 - G_2 \hat H_2 \\
\hat G_1 H_1 + \hat F_1 \hat K_1 + \hat G_2 H_2 + \hat F_2 \hat K_2 & -\hat G_1 K_1 + \hat F_1 \hat H_1 - \hat G_2 K_2 + \hat F_2 \hat H_2
\end{array}
\right)\\
& =\left(
\begin{array}{cc}
  1 & 0\\
  0  & 1
\end{array}
\right).
\end{align*}

Next, observe that
\begin{align*}
\mathsf{F}(z)\mathsf{F}^{\ast}(z) & =
\left(
\begin{array}{cccc}
F_1(z) & -G_1(z) & F_2(z) & -G_2(z) \\
\hat G_1(z) & \hat F_1(z) & \hat G_2(z) & \hat F_2(z)
\end{array}
\right)
\left(
\begin{array}{cc}
\overline{F_1(z)} & G_1(\bar z) \\
-\overline{G_1(z)} & F_1(\bar z) \\
\overline{F_2(z)} & G_2(\bar z) \\
-\overline{G_2(z)} & F_2(\bar z)
\end{array}
\right)\\
& = {
\left(\begin{array}{cc}
\vert F_1(z)\vert^2+\vert F_2(z)\vert^2 +\vert G_1(z)\vert^2+\vert G_2(z)\vert^2 & F_1 G_1(\bar z) - G_1 F_1(\bar z) + F_2 G_2(\bar z) - G_2 F_2(\bar z) \\
\overline{F_1 G_1(\bar z) - G_1 F_1(\bar z) + F_2 G_2(\bar z) - G_2 F_2(\bar z)} & \vert F_1(\bar z)\vert^2+\vert F_2(\bar z)\vert^2 +\vert G_1(\bar z)\vert^2+\vert G_2(\bar z)\vert^2
\end{array}
\right)}\\
& :=
\left(\begin{array}{cc}
a(z) & b(z)\\
\overline{b(z)} & d(z)
\end{array}
\right),
\end{align*}
where $d(z)=a(\bar z)$.
By Sylvester's Criterion, $\mathsf{F}(z)\mathsf{F}^{\ast}(z)\geq 0$ if and only if the determinants of all leading principal minors are non-negative. The Corona type condition, namely the typical hypothesis, is reflected in the diagonal terms, i.e., $0<\delta^2\leq\vert F_1(z)\vert^2+\vert F_2(z)\vert^2 +\vert G_1(z)\vert^2+\vert G_2(z)\vert^2\leq 1$, and so $0<\delta^2<a(z)\leq 1$ and similarly for $d(z)$.  It then suffices to check $\det( \mathsf{F}(z)\mathsf{F}^{\ast}(z))=a(z)d(z)-\vert b(z)\vert^2:=\Delta(z)> 0$.  A direct application of the Cauchy--Binet formula shows that the determinant of $\mathsf{F}(z)\mathsf{F}^{\ast}(z)$ is given by:
\begin{align}
\label{e:NewDenominator}
\Delta(z) & :=\abs{\hat{F}_1(z)F_2(z)+G_1(z)\hat{G}_2(z)}^2+\abs{F_1(z)\hat{G}_2(z)-F_2(z)\hat{G}_1(z)}^2
\\
&
+
\abs{F_2(z)\hat{F}_2(z)+G_2(z)\hat{G}_2(z)}^2
+\abs{G_1(z)\hat{F_2}(z)-\hat{F}_1(z)G_2(z)}^2
\notag\\
&
+\abs{F_1(z)\hat{F}_2(z)+\hat{G}_1(z)G_2(z)}^2+\abs{F_1(z)\hat{F}_1(z)+G_1(z)\hat{G}_1(z)}^2,\notag
\end{align}
and so the assertion follows. We defer the computation of this determinant to Theorem \ref{thm:detN-CB} in Appendix \ref{s:LinearAlgebra}, where it is derived via Cauchy--Binet.  In particular, the discussion shows that if we suppose that the data we start with has the Corona condition, then $\mathsf{F}(z)\mathsf{F}^*(z)\geq \delta^2\textnormal{Id}_2>0$.

By the matrix Corona Theorem applied to $\mathsf{F}(z)$ there exists a matrix function $\mathsf{P}(z)$,
$$
\mathsf{P}(z)=
\left(
\begin{array}{cc}
 p_{1,1}(z) & q_{1,1}(z) \\
 p_{2,1}(z) & q_{2,1}(z) \\
 p_{1,2}(z) & q_{1,2}(z) \\
 p_{2,2}(z) & q_{2,2}(z)
\end{array}
\right)
$$
so that, dropping the variable $z$ for the computations now,
\begin{align*}
\mathsf{F}\mathsf{P} & =\left(
\begin{array}{cccc}
F_1 & -G_1 & F_2 & -G_2 \\
\hat G_1 & \hat F_1 & \hat G_2 & \hat F_2
\end{array}
\right)
\left(
\begin{array}{cc}
 p_{1,1} & q_{1,1} \\
 p_{2,1} & q_{2,1} \\
 p_{1,2} & q_{1,2} \\
 p_{2,2} & q_{2,2}
\end{array}
\right)\\
&=
\left(
\begin{array}{cc}
 F_1 p_{1,1} - G_1 p_{2,1} + F_2 p_{1,2} - G_2 p_{2,2} & F_1 q_{1,1} - G_1 q_{2,1} + F_2 q_{1,2} - G_2 q_{2,2} \\
 \hat G_1 p_{1,1} + \hat F_1 p_{2,1} + \hat G_2 p_{1,2} + \hat F_2 p_{2,2} & \hat G_1 q_{1,1} + \hat F_1 q_{2,1} + \hat G_2 q_{1,2} + \hat F_2 q_{2,2}
\end{array}
\right)\\
& =
\left(
\begin{array}{cc}
 1 & 0 \\
 0 & 1
\end{array}
\right).
\end{align*}
This gives rise to four equations:
\begin{align}
 F_1 p_{1,1} - G_1 p_{2,1} + F_2 p_{1,2} - G_2 p_{2,2}  & = 1\label{e:s1e1}\\
 F_1 q_{1,1} - G_1 q_{2,1} + F_2 q_{1,2} - G_2 q_{2,2} & = 0\label{e:s1e2}\\
 \hat G_1 p_{1,1} + \hat F_1 p_{2,1} + \hat G_2 p_{1,2} + \hat F_2 p_{2,2} & = 0\label{e:s1e3}\\
 \hat G_1 q_{1,1} + \hat F_1 q_{2,1} + \hat G_2 q_{1,2} + \hat F_2 q_{2,2}  & = 1\label{e:s1e4}
\end{align}
and upon taking $\hat{\cdot}$ of these equations we have four additional equations:
\begin{align}
 \hat F_1 \hat p_{1,1} - \hat G_1 \hat p_{2,1} + \hat F_2 \hat p_{1,2} - \hat G_2 \hat p_{2,2}  & = 1\label{e:s2e1}\\
 \hat F_1 \hat q_{1,1} - \hat G_1 \hat q_{2,1} + \hat F_2 \hat q_{1,2} - \hat G_2 \hat q_{2,2} & = 0\label{e:s2e2}\\
 G_1 \hat p_{1,1} + F_1 \hat p_{2,1} + G_2 \hat p_{1,2} + F_2 \hat p_{2,2} & = 0\label{e:s2e3}\\
 G_1 \hat q_{1,1} + F_1 \hat q_{2,1} + G_2 \hat q_{1,2} + F_2 \hat q_{2,2}  & = 1.\label{e:s2e4}
\end{align}

If we add \eqref{e:s1e1} and \eqref{e:s2e4} and divide by two, we obtain
\begin{align}
\label{e:s3e1} \frac{(p_{1,1}+\hat{q}_{2,1})}{2}F_1+\frac{(\hat q_{1,1} - p_{2,1})}{2}G_1+\frac{(p_{1,2}+\hat{q}_{2,2})}{2}F_2+\frac{(\hat q_{1,2}-p_{2,2})}{2}G_2 & = 1.
\end{align}
Similarly, subtracting \eqref{e:s1e2} from \eqref{e:s2e3} and dividing by two gives:
\begin{align}
\label{e:s3e3} F_1\frac{(\hat p_{2,1} - q_{1,1})}{2}+G_1\frac{(\hat p_{1,1}+q_{2,1})}{2}+F_2\frac{(\hat p_{2,2} - q_{1,2})}{2}+G_2\frac{(\hat p_{1,2}+q_{2,2})}{2} & = 0.
\end{align}
These identities suggest defining:
\begin{align*}
H_1 & := \frac{(p_{1,1}+\hat{q}_{2,1})}{2} & H_2  & :=\frac{(p_{1,2}+\hat{q}_{2,2})}{2}\\
\hat{K}_1 & := \frac{(p_{2,1}-\hat q_{1,1})}{2} & \hat{K}_2 & := \frac{(p_{2,2}-\hat q_{1,2})}{2}\\
K_1 & := \frac{(\hat p_{2,1} - q_{1,1})}{2} & K_2 & :=  \frac{(\hat p_{2,2} - q_{1,2})}{2}\\
\hat{H}_1 & := \frac{(\hat{p}_{1,1}+q_{2,1})}{2} & \hat{H}_2  & :=\frac{(\hat{p}_{1,2}+q_{2,2})}{2}
\end{align*}
where the identities for $\hat{H}_j$ and $\hat{K}_j$ follow simply by taking $\hat{\cdot}$ of the definitions for $H_j$ and $K_j$.  Indeed, note that
$$
\hat K_j = \widehat{\left(\frac{\hat p_{2,j} - q_{1,j}}{2}\right)} = \frac{p_{2,j} - \hat q_{1,j}}{2} = -\frac{\hat q_{1,j} - p_{2,j}}{2},
$$
which matches the formula defined above.  If the functions $p_{i,j}$ and $q_{i,j}$ are bounded, then $H_j, K_j$ are bounded as well with the same norm.

With these functions, \eqref{e:s3e1} and \eqref{e:s3e3} translate into:
\begin{align}
\label{e:s3e1new} F_1H_1+F_2H_2-G_1\hat{K}_1-G_2\hat{K}_2  &= 1\\
\label{e:s3e3new} F_1K_1+F_2K_2+G_1\hat{H}_1+G_2\hat{H}_2 & =0.
\end{align}
Taking $\hat{\cdot}$ of \eqref{e:s3e1new} gives the further identity
\begin{align}
\label{e:s3e2new}  \hat{F}_1\hat{H}_1+\hat{F}_2\hat{H}_2-\hat{G}_1K_1-\hat{G}_2 K_2 &  = 1.
\end{align}

It then holds that \eqref{e:s3e1new}, \eqref{e:s3e2new} and \eqref{e:s3e3new} give:
$$
\left(
\begin{array}{cc}
  F_1H_1+F_2H_2-G_1\hat{K}_1-G_2\hat{K}_2
&   -F_1 K_1 - G_1 \hat H_1 - F_2 K_2 - G_2 \hat H_2 \\
\hat G_1 H_1 + \hat F_1 \hat K_1 + \hat G_2 H_2 + \hat F_2 \hat K_2 & \hat{F}_1\hat{H}_1+\hat{F}_2\hat{H}_2-\hat{G}_1K_1-\hat{G}_2 K_2
\end{array}
\right)=
\left(
\begin{array}{cc}
1 & 0\\
0 & 1
\end{array}\right)
$$
which means there is a matrix with bounded holomorphic entries
$$
\mathsf{H}(z)=
\left(
\begin{array}{cc}
H_1(z) & -K_1(z) \\
\hat K_1(z) & \hat H_1(z) \\
H_2(z) & -K_2(z) \\
\hat K_2(z) & \hat H_2(z)
\end{array}
\right)
$$
so that $\mathsf{F}(z)\mathsf{H}(z)=\textnormal{Id}_2$.

\subsection{The General Case}

Let $N\in \mathbb{N}\cup\{\infty\}$.  Suppose that we are given $F_j,G_j\in H^\infty(\mathbb{D})$ and $H_j, K_j\in H^\infty(\mathbb{D})$ for $1\leq j\leq N$ so that
\begin{align}
    \sum_{j=1}^{N}F_jH_j-G_j\hat{K}_j &=1\label{e:e1n}\\
    \sum_{j=1}^{N} F_jK_j+G_j\hat{H}_j &= 0\label{e:e2n}.
\end{align}
For the given $F_j,G_j\in H^\infty(\mathbb{D})$ define
$$
\mathsf{F}(z)=
\left(
\begin{array}{cccccc}
F_1(z) & -G_1(z) & \cdots & F_N(z) & -G_N(z) \\
\hat G_1(z) & \hat F_1(z) & \cdots & \hat G_N(z) & \hat F_N(z)
\end{array}
\right)
$$
and
$$
\mathsf{H}(z)=
\left(
\begin{array}{cc}
H_1(z) & -K_1(z) \\
\hat K_1(z) & \hat H_1(z) \\
\vdots & \vdots \\
H_N(z) & -K_N(z) \\
\hat K_N(z) & \hat H_N(z)
\end{array}
\right).
$$
The matrix $\mathsf{F}(z)$ is a $2\times 2N$ matrix and $\mathsf{H}(z)$ is a $2N\times 2$ matrix.  The columns of $\mathsf{F}(z)$ are arranged in an \emph{interleaved} fashion: for each $1\leq j\leq N$, the columns $(F_j,\hat G_j)^\top$ and $(-G_j,\hat F_j)^\top$ associated to the generator $j$ appear consecutively (where $^\top$ denotes the transpose).  Correspondingly, the rows of $\mathsf{H}(z)$ are arranged so that the entries associated to the $j$-th generator appear in rows $2j-1, 2j$.  This interleaved arrangement will be convenient for the developments in later sections.

Observe that, dropping the variable $z$ for the computation
\begin{align*}
\mathsf{F}\mathsf{H} & =
\left(
\begin{array}{cc}
  \displaystyle\sum_{j=1}^{N}F_jH_j-G_j\hat{K}_j
& \displaystyle -\sum_{j=1}^{N}(F_jK_j+G_j\hat{H}_j) \\
\displaystyle\sum_{j=1}^{N}\hat{G}_jH_j+\hat{F}_j\hat{K}_j & \displaystyle\sum_{j=1}^{N}\hat{F}_j\hat{H}_j-\hat{G}_jK_j
\end{array}
\right) =\left(
\begin{array}{cc}
  1 & 0\\
  0  & 1
\end{array}
\right).
\end{align*}
Here the entries of the resulting $2\times 2$ matrix follow from either using \eqref{e:e1n} and \eqref{e:e2n} directly, taking the $\widehat{\cdot}$ of the equation, or multiplying by $-1$.

The goal is to solve the equations \eqref{e:e1n} and \eqref{e:e2n} under a Corona condition on the data $F_j, G_j$.  We now argue how this can be done via the matrix Corona Theorem for $H^\infty(\mathbb{D})$ working from the matrix $\mathsf{F}(z)$ constructed above.

Observe that
\begin{align*}
\mathsf{F}(z)\mathsf{F}^{\ast}(z) & =
\left(\begin{array}{cc}
\displaystyle\sum_{j=1}^{N} \vert F_j(z)\vert^2+\vert G_j(z)\vert^2 & \displaystyle\sum_{j=1}^{N}F_j(z)G_j(\overline{z})-G_j(z)F_j(\overline{z})\\
\displaystyle\sum_{j=1}^{N}\overline{F_j(z)G_j(\overline{z})}-\overline{G_j(z)F_j(\overline{z})} & \displaystyle\sum_{j=1}^{N} \vert F_j(\overline{z})\vert^2 +\vert G_j(\overline{z})\vert^2
\end{array}
\right)\\
& :=
\left(\begin{array}{cc}
a(z) & b(z)\\
\overline{b(z)} & d(z)
\end{array}
\right).
\end{align*}

As in the case of two generators, $\mathsf{F}\mathsf{F}^*$ is a Gram matrix and by
 Sylvester's Criterion, to show that $\mathsf{F}(z)\mathsf{F}^{\ast}(z)> 0$ it suffices to verify that the determinants leading principal minors are non-negative.  Clearly the leading principal minor is non-negative by the usual Corona condition and similarly for the other diagonal element.  And by Theorem \ref{thm:detN-CB} below we have that $\Delta(z)=\det \mathsf{F}(z)\mathsf{F}^{\ast}(z)$ is given by:
\begin{align*}
\Delta(z) &=\sum_{r=1}^{N}\sum_{j=1}^{N}
\left\vert F_j(z)\hat{F}_r(z)+G_r(z)\hat{G}_j(z)\right\vert^2
+\sum_{r=1}^{N}\sum_{j=r+1}^{N}\left\vert \hat{G}_r(z)F_j(z)-\hat{G}_j(z)F_r(z)\right\vert^2\\
& \quad +\sum_{r=1}^{N}\sum_{j=r+1}^{N}\left\vert \hat{F}_r(z)G_j(z)-\hat{F}_j(z)G_r(z)\right\vert^2.
\end{align*}
See Section \ref{s:LinearAlgebra} for the justification of these computations.  And so if we start with data that has the Corona condition, then $\mathsf{F}(z)\mathsf{F}^*(z)\geq \delta^2\textnormal{Id}_2>0$ as we shall discuss in Section \ref{subsec:Bezout-general}.

By the $H^\infty(\mathbb{D})$ matrix Corona Theorem applied to $\mathsf{F}(z)$ there exists a matrix function $\mathsf{P}(z)$, written in interleaved form to match the column-ordering of $\mathsf{F}(z)$,
$$
\mathsf{P}(z)=
\left(
\begin{array}{cc}
 p_{1,1}(z) & q_{1,1}(z) \\
 p_{2,1}(z) & q_{2,1}(z) \\
 \vdots & \vdots \\
 p_{1,N}(z) & q_{1,N}(z) \\
 p_{2,N}(z) & q_{2,N}(z)
\end{array}
\right)
$$
so that $\mathsf{F}(z)\mathsf{P}(z)=\textnormal{Id}_2$.  Here the rows of $\mathsf{P}(z)$ are indexed in such a way that, for each $1\le j\le N$, the entries $p_{i,j}$ pair with the $(F_j,\hat{G}_j)^\top$ column of $\mathsf{F}$ and the entries $q_{i,j}$ pair with the $(-G_j,\hat{F}_j)^\top$ column.

Direct multiplication gives
\begin{align*}
\mathsf{F}\mathsf{P} & =
\left(
\begin{array}{cc}
 \displaystyle\sum_{j=1}^{N}F_j p_{1,j}-G_j p_{2,j} & \displaystyle\sum_{j=1}^{N}F_j q_{1,j} - G_j q_{2,j} \\
\displaystyle \sum_{j=1}^{N}\hat G_j p_{1,j} + \hat F_j p_{2,j}  & \displaystyle\sum_{j=1}^{N}\hat G_j q_{1,j} + \hat F_j q_{2,j}
\end{array}
\right)=
\left(
\begin{array}{cc}
 1 & 0 \\
 0 & 1
\end{array}
\right).
\end{align*}
We want to use the functions $p_{i,j}$ and $q_{i,j}$ to build functions $H_j$ and $K_j$ that solve \eqref{e:e1n} and \eqref{e:e2n}.
The equation $\mathsf{F}(z)\mathsf{P}(z)=\textnormal{Id}_2$ gives rise to four equations:
\begin{align}
 \sum_{j=1}^{N} F_j p_{1,j} - G_j p_{2,j}   & = 1;\label{e:s1e1n}\\
 \sum_{j=1}^{N} F_j q_{1,j} - G_j q_{2,j} & = 0;\label{e:s1e2n}\\
 \sum_{j=1}^{N} \hat G_j p_{1,j} + \hat F_j p_{2,j} & = 0;\label{e:s1e3n}\\
 \sum_{j=1}^{N} \hat G_j q_{1,j} + \hat F_j q_{2,j}   & = 1;\label{e:s1e4n}
\end{align}
and upon taking $\hat{\cdot}$ of these equations we have four additional equations:
\begin{align}
 \sum_{j=1}^{N}\hat F_j \hat p_{1,j} - \hat G_j \hat p_{2,j}  & = 1;\label{e:s2e1n}\\
 \sum_{j=1}^{N}\hat F_j \hat q_{1,j} - \hat G_j \hat q_{2,j} & = 0;\label{e:s2e2n}\\
 \sum_{j=1}^{N} G_j \hat p_{1,j} + F_j \hat p_{2,j} & = 0;\label{e:s2e3n}\\
 \sum_{j=1}^{N} G_j \hat q_{1,j} + F_j \hat q_{2,j}  & = 1.\label{e:s2e4n}
\end{align}
If we add \eqref{e:s1e1n} and \eqref{e:s2e4n} and divide by two we obtain:
\begin{align}
\label{e:s3e1n}\sum_{j=1}^{N}\frac{(p_{1,j}+\hat{q}_{2,j})}{2}F_j+\frac{(\hat q_{1,j} - p_{2,j})}{2}G_j  & = 1.
\end{align}
Then subtracting the equation \eqref{e:s1e2n} from the equation \eqref{e:s2e3n}, collecting terms and dividing by two we see that:
\begin{align}
\label{e:s3e3n}  \sum_{j=1}^{N}F_j\frac{(\hat p_{2,j} - q_{1,j})}{2} + G_j\frac{(\hat{p}_{1,j}+q_{2,j})}{2} & =0.
\end{align}
These identities suggest defining the following functions:
\begin{align*}
H_j & := \frac{(p_{1,j}+\hat{q}_{2,j})}{2} & \hat{H}_j & = \frac{(\hat{p}_{1,j}+q_{2,j})}{2}\\
 K_j & := \frac{(\hat p_{2,j} - q_{1,j})}{2} & \hat{K}_j & = \frac{(p_{2,j} - \hat q_{1,j})}{2}
\end{align*}
with the resulting identities for $\hat{H}_j$ and $\hat{K}_j$ following simply by taking $\hat{\cdot}$ of the definition for $H_j$ and $K_j$.  Also, observe that if the functions $p_{i,j}$ and $q_{i,j}$ are bounded for $i=1,2$ and $1\leq j\leq N$, then it will be the case that $H_j, K_j$ are bounded as well, with the same norm.

With these definitions of $H_j$ and $K_j$, we interpret \eqref{e:s3e1n} and \eqref{e:s3e3n} as
\begin{align}
\label{e:s3e1nnew} \sum_{j=1}^{N}F_jH_j-G_j\hat{K}_j &= 1,\\
\label{e:s3e3nnew} \sum_{j=1}^{N} F_jK_j+G_j\hat{H}_j & =0.
\end{align}
Taking $\hat{\cdot}$ of \eqref{e:s3e1nnew} gives the further identity
\begin{align}
\label{e:s3e2nnew}  \sum_{j=1}^{N} \hat{F}_j\hat{H}_j-\hat{G}_jK_j &  = 1.
\end{align}

It then holds that \eqref{e:s3e1nnew}, \eqref{e:s3e3nnew} and \eqref{e:s3e2nnew} give:
$$
\left(
\begin{array}{cc}
  \displaystyle\sum_{j=1}^{N}F_jH_j-G_j\hat{K}_j & \displaystyle-\sum_{j=1}^{N}(F_jK_j+G_j\hat{H}_j) \\
\displaystyle\sum_{j=1}^{N}\hat{G}_jH_j+\hat{F}_j\hat{K}_j & \displaystyle\sum_{j=1}^{N}\hat{F}_j\hat{H}_j-\hat{G}_jK_j
\end{array}
\right)=
\left(
\begin{array}{cc}
1 & 0\\
0 & 1
\end{array}\right);
$$
which means there is a matrix with bounded holomorphic entries
$$
\mathsf{H}(z)=
\left(
\begin{array}{cc}
H_1(z) & -K_1(z) \\
\hat K_1(z) & \hat H_1(z) \\
\vdots & \vdots \\
H_N(z) & -K_N(z) \\
\hat K_N(z) & \hat H_N(z)
\end{array}
\right)
$$
so that $\mathsf{F}(z)\mathsf{H}(z)=\textnormal{Id}_2$.  Rephrasing it yet again, \eqref{e:e1n} and \eqref{e:e2n} hold for the solutions $H_j$ and $K_j$ when $F_j, G_j$ satisfy a Corona condition.

\subsection{Matrix-block computations}
\label{ss:matrix-block}

We now give a more linear-algebraic reformulation of the construction in the previous subsection.  The reformulation has two advantages: it makes manifest the role of an underlying symmetry of $\mathsf{F}$ (induced by the involution $\hat{\cdot}$), and it shows that the symmetrization decouples into a blockwise operation, one block per generator.  Both features will be useful in the Clifford setting later on.

Throughout this subsection, $\mathsf{F}(z)$ is the $2 \times 2N$ matrix of the previous subsection in its \emph{interleaved} form.  It will be convenient to view $\mathsf{F}(z)$ as a row of $N$ blocks of size $2\times 2$:
$$
\mathsf{F}(z)=
\bigl(F^{(1)}(z)\ \cdots\ F^{(N)}(z)\bigr),
\qquad
F^{(j)}(z)=
\begin{pmatrix}
F_j(z) & -G_j(z)\\
\hat G_j(z) & \hat{F}_j(z)
\end{pmatrix}
= M_{F_j}(z).
$$

Define the $2 \times 2$ matrix
$$
J_2=
\begin{pmatrix}
0 & 1\\
-1 & 0
\end{pmatrix}
$$
and the $2N \times 2N$ block-diagonal matrix
$$
J_{2N}=\textnormal{Id}_N\otimes J_2 = \operatorname{diag}(J_2,\dots,J_2)\in \mathbb{R}^{2N\times 2N}
$$
with $N$ copies of $J_2$ on the diagonal.  Note that $J_2^{-1} = -J_2 = J_2^\top$ and $J_{2N}^{-1} = -J_{2N} = J_{2N}^\top$.

A direct blockwise computation shows that, for each $j$,
$$
\widehat{F^{(j)}}=J_2\, F^{(j)}\, J_2^{-1},
$$
and consequently
\begin{equation}
\label{eq:symmetry-F}
\widehat{\mathsf{F}}(z)=J_2\,\mathsf{F}(z)\,J_{2N}^{-1}.
\end{equation}

Suppose that the matrix Corona Theorem yields a right inverse
$$
\mathsf{P}(z)\in H^\infty(\mathbb D)^{2N\times 2},
\qquad
\mathsf{F}(z)\mathsf{P}(z)=\textnormal{Id}_2.
$$
Writing $\mathsf{P}(z)$ in the interleaved form of the previous subsection, it is convenient to view $\mathsf{P}(z)$ as a column of $N$ blocks of size $2 \times 2$:
$$
\mathsf{P}(z)=
\begin{pmatrix}
P^{(1)}(z) \\
\vdots\\
P^{(N)}(z)
\end{pmatrix},
\qquad
P^{(j)}(z)=
\begin{pmatrix}
p_{1,j}(z) & q_{1,j}(z)\\
p_{2,j}(z) & q_{2,j}(z)
\end{pmatrix}.
$$

Taking $\hat{\cdot}$ of the identity $\mathsf{F}(z) \mathsf{P}(z)=\textnormal{Id}_2$, using \eqref{eq:symmetry-F}, we have
$$
\hat{\mathsf{F}}(z)\hat{\mathsf{P}}(z) = \textnormal{Id}_2
\quad\Longrightarrow\quad
J_2\,\mathsf{F}(z)\,J_{2N}^{-1}\,\hat{\mathsf{P}}(z) = \textnormal{Id}_2
\quad\Longrightarrow\quad
\mathsf{F}(z)\,J_{2N}^{-1}\,\hat{\mathsf{P}}(z)\,J_2 = J_2^{-1}\,\textnormal{Id}_2\,J_2 = \textnormal{Id}_2.
$$
Hence the averaged matrix
\begin{equation}
\label{eq:def-H-avg}
\mathsf{H}(z):=\frac12\Big(\mathsf{P}(z)+J_{2N}^{-1}\widehat{\mathsf{P}}(z)J_2\Big)
\end{equation}
also satisfies
$$
\mathsf{F}(z) \mathsf{H}(z)=\textnormal{Id}_2.
$$

Because $J_{2N}$ is block-diagonal, the symmetrization \eqref{eq:def-H-avg} decouples across the blocks $P^{(j)}$: writing $\mathsf{H}(z)$ as block column with the $2 \times 2$ blocks $H^{(1)}(z), \ldots H^{(N)}(z)$ stacked vertically, we have
$$
H^{(j)}(z) = \frac12\bigl(P^{(j)}(z) + J_2^{-1} \widehat{P^{(j)}}(z) J_2\bigr).
$$
A direct computation gives
$$
J_2^{-1}\,\widehat{P^{(j)}}(z)\,J_2=
\begin{pmatrix}
\hat{q}_{2,j}(z) & -\hat{p}_{2,j}(z)\\
-\hat{q}_{1,j}(z) & \hat{p}_{1,j}(z)
\end{pmatrix},
$$
and so
$$
H^{(j)}(z)=
\frac12
\begin{pmatrix}
p_{1,j}+\hat{q}_{2,j} & q_{1,j}-\hat{p}_{2,j}\\
p_{2,j}-\hat{q}_{1,j} & q_{2,j}+\hat{p}_{1,j}
\end{pmatrix}=
\begin{pmatrix}
H_j & -K_j\\
\hat K_j & \hat{H}_j
\end{pmatrix},
$$
where
$$
H_j=\frac{p_{1,j}+\hat{q}_{2,j}}{2},
\qquad
K_j=\frac{\hat p_{2,j}-q_{1,j}}{2}.
$$
The functions $\hat{H}_j$ and $\hat{K}_j$ are then obtained from $H_j$ and $K_j$ by applying the involution $\hat{\cdot}$.

These are exactly the same formulas as derived coordinatewise in the previous subsection.  Assembling the blocks $H^{(j)}$ back into $\mathsf{H}(z)$,
$$
\mathsf{H}(z)=
\left(
\begin{array}{cc}
H_1(z) & -K_1(z) \\
\hat K_1(z) & \hat H_1(z) \\
\vdots & \vdots \\
H_N(z) & -K_N(z) \\
\hat K_N(z) & \hat H_N(z)
\end{array}
\right),
$$
which is the same matrix $\mathsf{H}(z)$ as in the previous subsection.  Since $\mathsf{F}(z) \mathsf{H}(z)=\textnormal{Id}_2$, we obtain
\begin{align*}
\sum_{j=1}^N F_j H_j-G_j\widehat{K}_j&=1,\\
\sum_{j=1}^N F_j K_j+G_j\widehat{H}_j&=0.
\end{align*}

\section{\texorpdfstring{General Formula for the $*$-Product of Slice Monogenic Functions}{General Formula for the *-Product of Slice Monogenic Functions}}
\label{sec:slice-monogenic-preliminaries}

\subsection{Preliminaries on slice monogenic functions}

We collect here the basic definitions and structural results for slice monogenic functions over real Clifford algebras, following Colombo--Sabadini--Struppa~\cite{MR2684426}.

\subsubsection{The Clifford algebra $\mathbb{R}_n$}

The real Clifford algebra $\mathbb{R}_n$ is the associative algebra over $\mathbb{R}$ generated by $n$ basis elements $I_1, \ldots, I_n$, called the \emph{imaginary units}, subject to the defining relations
\begin{equation}\label{eq:clifford-relations}
I_i I_j + I_j I_i = -2 \delta_{ij}.
\end{equation}
A generic element $x \in \mathbb{R}_n$ admits the expansion
\[
x = \sum_{|A| = 0}^{n} x_A \, I_A,
\]
where $A = i_1 \cdots i_t$ is an increasing multi-index of length $t = |A|$, $I_A := I_{i_1} \cdots I_{i_t}$, and $I_\emptyset := 1$.

An $n$-tuple $(x_1, \ldots, x_n) \in \mathbb{R}^n$ is identified with the $1$-vector $\underline x := x_1 I_1 + \cdots + x_n I_n$, while real numbers are identified with $0$-vectors. With a slight abuse of notation, elements
\[
\vec x = x_0 + x_1 I_1 + \cdots + x_n I_n = x_0 + \underline x \in \mathbb{R}_n^0 \oplus \mathbb{R}_n^1
\]
are identified with $(x_0, x_1, \ldots, x_n) \in \mathbb{R}^{n+1}$, and we write $\vec x \in \mathbb{R}^{n+1}$, $\underline x \in \mathbb{R}^n$. The norm of $\vec x$ is $|\vec x| = (x_0^2 + \cdots + x_n^2)^{1/2}$.

\subsubsection{The sphere $\mathbb{S}$ of unit $1$-vectors}

We denote by $\mathbb{S}$ the sphere of unit $1$-vectors in $\mathbb{R}^{n+1}$,
\[
\mathbb{S} := \bigl\{\, \underline x = I_1 x_1 + \cdots + I_n x_n \in \mathbb{R}^{n+1} \,\bigm|\, x_1^2 + \cdots + x_n^2 = 1 \,\bigr\}.
\]
For any $I \in \mathbb{S}$, the Clifford relations give $I^2 = -1$. The two-dimensional real subspace
\[
\mathbb{C}_I := \mathbb{R} + I \, \mathbb{R}
\]
is algebra-isomorphic to the complex plane $\mathbb{C}$. We refer to $\mathbb{C}_I$ as a \emph{complex plane}; elements of $\mathbb{C}_I$ are written $z = u + Iv$ with $u, v \in \mathbb{R}$.

\subsubsection{Slice monogenic functions}

\begin{definition}[Slice monogenic function]\label{def:slice-monogenic}
Let $U \subseteq \mathbb{R}^{n+1}$ be an open set. A real-differentiable function $f \colon U \to \mathbb{R}_n$ is called \emph{(left) slice monogenic} (or \emph{$s$-monogenic} for short) if, for every $I \in \mathbb{S}$, the restriction $f_I := f|_{U \cap \mathbb{C}_I}$ satisfies
\begin{equation}\label{eq:s-monogenic-def}
\frac{1}{2}\Bigl(\frac{\partial}{\partial u} + I \frac{\partial}{\partial v}\Bigr) f_I(u + Iv) = 0.
\end{equation}
We denote by $\mathcal M(U)$ the set of $s$-monogenic functions on $U$.
\end{definition}

\noindent We write \eqref{eq:s-monogenic-def} compactly as $\overline\partial_I f_I = 0$. The corresponding notion of right $s$-monogenicity is defined by $f_I \overline \partial_I = 0$; the theory is symmetric, and we restrict attention to left $s$-monogenic functions throughout.

\begin{definition}[$s$-derivative]\label{def:s-derivative}
The \emph{$I$-derivative} is defined by
\[
\partial_I := \frac{1}{2}\Bigl(\frac{\partial}{\partial u} - I \frac{\partial}{\partial v}\Bigr).
\]
For $f \in \mathcal M(U)$, the \emph{$s$-derivative} of $f$ is
\[
\partial_s f(\vec x) :=
\begin{cases}
\partial_I f(\vec x), & \vec x = u + Iv,\ v \ne 0,\\
\partial_u f(u), & u \in \mathbb{R}.
\end{cases}
\]
It coincides with $\partial_u f_I$ on each $\mathbb{C}_I$.
\end{definition}

\subsubsection{The Splitting Lemma}

\begin{lemma}[Splitting Lemma {\cite{MR2684426}*{Lemma~2.4}}]\label{lem:splitting}
Let $U \subseteq \mathbb{R}^{n+1}$ be open and let $f \in \mathcal M(U)$. For every $I = I_1 \in \mathbb{S}$, let $I_2, \ldots, I_n$ be a completion to a basis of $\mathbb{R}_n$ such that $I_i I_j + I_j I_i = -2\delta_{ij}$. Then there exist $2^{n-1}$ holomorphic functions
\[
F_A \colon U \cap \mathbb{C}_I \to \mathbb{C}_I, \qquad A \subseteq \{2, \ldots, n\},
\]
such that for every $z = u + Iv \in U \cap \mathbb{C}_I$,
\begin{equation}\label{eq:splitting}
f_I(z) = \sum_{|A| = 0}^{n-1} F_A(z)\, I_A,
\end{equation}
where $I_A := I_{i_1} \cdots I_{i_s}$ for $A = i_1 \cdots i_s$ with $2\leq i_1 < \cdots < i_s\leq n$, and $I_\emptyset := 1$.
\end{lemma}

\subsubsection{Slice domains and axial symmetry}

\begin{definition}[Slice domain, axial symmetry]\label{def:slice-domain}
A domain $U \subseteq \mathbb{R}^{n+1}$ is called a \emph{slice domain} (or \emph{$s$-domain}) if $U \cap \mathbb{R} \ne \emptyset$ and $U \cap \mathbb{C}_I$ is a domain in $\mathbb{C}_I$ for every $I \in \mathbb{S}$. The domain $U$ is \emph{axially symmetric} if, for every $u + Iv \in U$, the entire $(n-1)$-sphere $\{u + Jv : J \in \mathbb{S}\}$ is contained in $U$.
\end{definition}

Unless otherwise stated, we restrict our attention to axially symmetric slice domains.

\subsubsection{The Identity Principle and power series}

\begin{thm}[Identity Principle {\cite{MR2684426}*{Theorem~2.6}}]\label{thm:identity-principle}
Let $U \subseteq \mathbb{R}^{n+1}$ be an $s$-domain, let $f \colon U \to \mathbb{R}_n$ be $s$-monogenic, and let $Z := \{\vec x \in U : f(\vec x) = 0\}$. If there exists $I \in \mathbb{S}$ such that $\mathbb{C}_I \cap Z$ has an accumulation point in $U \cap \mathbb{C}_I$, then $f \equiv 0$ on $U$.
\end{thm}

\begin{prop}[Power-series expansion at a real point {\cite{MR2684426}*{Proposition~2.7}}]\label{prop:power-series}
Let $f$ be $s$-monogenic on an $s$-domain $U \subseteq \mathbb{R}^{n+1}$, and let $y_0 \in U \cap \mathbb{R}$. Then $f$ admits the power series representation
\[
f(\vec x) = \sum_{k=0}^{\infty} (\vec x - y_0)^k \, \frac{1}{k!}\frac{\partial^k f}{\partial u^k}(y_0)
\]
on the ball $B(y_0, R) \subseteq U$, where $R$ is the largest positive real number with $B(y_0, R) \subseteq U$.
\end{prop}

\subsubsection{The Representation Formula and Extension Theorem}

For $\vec x = x_0 + \underline x \in \mathbb{R}^{n+1}$, we set
\[
I_{\vec x} := \begin{cases} \dfrac{\underline x}{|\underline x|}, & \underline x \ne 0,\\ \text{any element of } \mathbb{S}, & \text{otherwise.} \end{cases}
\]

\begin{thm}[Representation Formula {\cite{MR2684426}*{Theorem~3.1}}]\label{thm:representation}
Let $U \subseteq \mathbb{R}^{n+1}$ be an axially symmetric $s$-domain and let $f \in \mathcal M(U)$. For any vector $\vec x = x_0 + I_{\vec x} |\underline x| \in U$ and any choice of $I \in \mathbb{S}$,
\begin{equation}\label{eq:representation}
f(\vec x) = \frac{1}{2}\bigl[f(x_0 + I|\underline x|) + f(x_0 - I|\underline x|)\bigr] + \frac{1}{2}\, I_{\vec x}\, I\, \bigl[f(x_0 - I|\underline x|) - f(x_0 + I|\underline x|)\bigr].
\end{equation}
\end{thm}

\begin{cor}[Slice-independence of the symmetric/antisymmetric parts {\cite{MR2684426}*{Corollary~3.2}}]\label{cor:slice-independence}
Let $U \subseteq \mathbb{R}^{n+1}$ be an axially symmetric $s$-domain and let $f \in \mathcal M(U)$. The quantities
\[
\tfrac{1}{2}\bigl[f(x_0 + I|\underline x|) + f(x_0 - I|\underline x|)\bigr] \qquad\text{and}\qquad I\, \tfrac{1}{2}\bigl[f(x_0 - I|\underline x|) - f(x_0 + I|\underline x|)\bigr]
\]
are independent of $I \in \mathbb{S}$.
\end{cor}

\begin{thm}[Extension Theorem {\cite{MR2684426}*{Theorem~3.3}}]\label{thm:extension}
Let $J \in \mathbb{S}$ and let $D$ be a domain in $\mathbb{C}_J$ that is symmetric with respect to the real axis and satisfies $D \cap \mathbb{R} \ne \emptyset$. Define the axially symmetric $s$-domain
\[
U_D := \bigcup_{u + Jv \in D,\, I \in \mathbb{S}} (u + Iv) \subseteq \mathbb{R}^{n+1}.
\]
If $f \colon D \to \mathbb{C}_J$ is holomorphic, then the function $\mathrm{ext}(f) \colon U_D \to \mathbb{R}_n$ defined by
\begin{equation}\label{eq:extension-formula}
\mathrm{ext}(f)(u + Iv) := \tfrac{1}{2}\bigl[f(u + Jv) + f(u - Jv)\bigr] + I \tfrac{1}{2}\, J\, \bigl[f(u - Jv) - f(u + Jv)\bigr]
\end{equation}
is the unique $s$-monogenic extension of $f$ to $U_D$. The same conclusion holds if $f \colon D \to \mathbb{R}_n$ is merely assumed to satisfy $\overline\partial_J f(u + Jv) = 0$.
\end{thm}

\begin{thm}[Smoothness {\cite{MR2684426}*{Theorem~3.4}}]\label{thm:smoothness}
Every $s$-monogenic function $f \colon U \to \mathbb{R}_n$ on an axially symmetric $s$-domain $U \subseteq \mathbb{R}^{n+1}$ is infinitely differentiable on $U$.
\end{thm}

\begin{prop}[$s$-derivative is $s$-monogenic {\cite{MR2684426}*{Proposition~3.5}}]\label{prop:s-derivative-monogenic}
Let $U \subseteq \mathbb{R}^{n+1}$ be axially symmetric and $f \in \mathcal M(U)$. Then $\partial_s f \in \mathcal M(U)$, and
\[
\partial_s^n f(u + Iv) = \frac{\partial^n f}{\partial u^n}(u + Iv).
\]
\end{prop}

\begin{cor}[Geometric structure of the image {\cite{MR2684426}*{Corollary~3.6}}]\label{cor:image-sphere}
Let $U \subseteq \mathbb{R}^{n+1}$ be an axially symmetric $s$-domain and let $f \colon U \to \mathbb{R}_n$ be $s$-monogenic. For all $u, v \in \mathbb{R}$ such that $u + Iv \in U$, there exist $a, b \in \mathbb{R}_n$ depending only on $u, v$ such that
\begin{equation}\label{eq:image-sphere}
f(u + Iv) = a + Ib \qquad \text{for every } I \in \mathbb{S}.
\end{equation}
In particular, the image of the $(n-1)$-sphere $\{u + Iv : I \in \mathbb{S}\}$ under $f$ is $\{a + Ib : I \in \mathbb{S}\}$.
\end{cor}

\subsection{\texorpdfstring{The $\ast$-product}{The *-product}}

We next collect a useful formula for the $\ast$-product that improves upon the work \cite{MR2684426}.  We set the stage for this result now.

Let $\{I_1, I_2, \ldots, I_n\}$ be generators of the Clifford algebra $\mathbb{R}_n$ satisfying
\[
I_i I_j + I_j I_i = -2\delta_{ij}.
\]
Fix $I = I_1$ and let $\mathbb{C}_I = \{u + Iv : u, v \in \mathbb{R}\}$ be the associated complex plane.  For a subset $A \subseteq \{2, \ldots, n\}$ with elements $A = \{a_1 < a_2 < \cdots < a_k\}$, define
\[
I_A = I_{a_1} I_{a_2} \cdots I_{a_k}, \qquad I_\emptyset = 1.
\]

Let $U \subseteq \mathbb{R}^{n+1}$ be an axially symmetric $s$-domain and let $f, g : U \to \mathbb{R}_n$ be $s$-monogenic functions. For any $I \in \mathbb{S}$ set $I = I_1$ and consider a completion to a basis $\{I_1, \ldots, I_n\}$ of $\mathbb{R}_n$ such that $I_i I_j + I_j I_i = -2\delta_{ij}$. The Splitting Lemma guarantees the existence of holomorphic functions $F_A, G_A : U \cap \mathbb{C}_I \to \mathbb{C}_I$ such that, for all $z = u + Iv \in U \cap \mathbb{C}_I$,
\[
f_I(z) = \sum_A F_A(z) I_A, \qquad g_I(z) = \sum_B G_B(z) I_B,
\]
where $A, B$ are subsets of $\{2, \ldots, n\}$.

We define the function $f_I * g_I : U \cap \mathbb{C}_I \to \mathbb{R}_n$, giving the $\ast$-product, as
\begin{align}
(f_I * g_I)(z) &= \sum_{|A| \text{ even}} (-1)^{\frac{|A|}{2}} F_A(z) G_A(z) + \sum_{|A| \text{ odd}} (-1)^{\frac{|A|+1}{2}} F_A(z) \overline{G_A(\bar{z})} \notag \\
&\quad + \sum_{\substack{|A| \text{ even} \\ B \neq A}} F_A(z) G_B(z) I_A I_B + \sum_{\substack{|A| \text{ odd} \\ B \neq A}} F_A(z) \overline{G_B(\bar{z})} I_A I_B.\label{e:sliceproduct_CSS}
\end{align}
This is the $\ast$ product introduced in \cite{MR2684426}.  By definition we have that $f*g:\ U\to\mathbb{R}_n$ is
$(f*g):={\rm ext}(f_I*g_I)$.
If we define the conjugation operator $\hat{\cdot}$ acting on functions by
$
\hat{G}(z) = \overline{G(\bar{z})}
$,
then we can write \eqref{e:sliceproduct_CSS} in the following equivalent form:
\begin{align}
(f_I * g_I)(z) &= \sum_{|A| \text{ even}} (-1)^{\frac{|A|}{2}} F_A(z) G_A(z) + \sum_{|A| \text{ odd}} (-1)^{\frac{|A|+1}{2}} F_A(z) \hat{G}_A(z) \notag \\
&\quad + \sum_{\substack{|A| \text{ even} \\ B \neq A}} F_A(z) G_B(z) I_A I_B + \sum_{\substack{|A| \text{ odd} \\ B \neq A}} F_A(z) \hat{G}_B(z) I_A I_B. \label{e:sliceproduct_CSS2}
\end{align}

From equation \eqref{e:sliceproduct_CSS} or \eqref{e:sliceproduct_CSS2} we can write the $\ast$-product in the following way:
\[
(f_I * g_I)(z) = \underbrace{\sum_{|A| \text{ even}} (-1)^{\frac{|A|}{2}} F_A(z) G_A(z) + \sum_{|A| \text{ odd}} (-1)^{\frac{|A|+1}{2}} F_A(z) \hat{G}_A(z)}_{\text{diagonal terms } (A = B)} + \underbrace{\sum_{A \neq B} F_A(z) \tilde{G}_B(z) I_A I_B}_{\text{off-diagonal terms}}.
\]
where
\[
\tilde{G}_B(z) =
\begin{cases}
G_B(z) & \text{if } |A| \text{ is even}, \\
\hat{G}_B(z) = \overline{G_B(\bar{z})} & \text{if } |A| \text{ is odd}.
\end{cases}
\]

\subsection{The General Formula}

The goal is to give an alternate formulation of \eqref{e:sliceproduct_CSS} or \eqref{e:sliceproduct_CSS2} where we collapse the sums over $A\neq B$ into a single sum.  This entails cancellations in the product $I_AI_B$ when there are common entries in $A$ and $B$ and reordering the product to a standard form.  We do this now and explain the computations.

Let $f_I, g_I : U \cap \mathbb{C}_I \to \mathbb{R}_n$ be holomorphic functions with expansions
\[
f_I(z) = \sum_{A \subseteq \{2,\ldots,n\}} F_A(z) I_A, \qquad g_I(z) = \sum_{B \subseteq \{2,\ldots,n\}} G_B(z) I_B,
\]
where $F_A, G_B : U \cap \mathbb{C}_I \to \mathbb{C}_I$ are $\mathbb{C}_I$-valued holomorphic functions.

Then the $*$-product is given by
\begin{equation}
\label{eq:newproduct}
(f_I * g_I)(z) = \sum_{C \subseteq \{2,\ldots,n\}} H_C(z) \, I_C
\end{equation}
where the coefficient $H_C$ is the function given by:
\begin{equation}
\label{eq:newcoefficients}
H_C(z) := \sum_{\substack{A, B \subseteq \{2,\ldots,n\} \\ A \triangle B = C}} \text{sign}(A,B) \cdot F_A(z) \cdot \tilde{G}_B(z)
\end{equation}
with the rules for the sign and conjugation are defined in the following manner.  The conjugation rule is
\[
\tilde{G}_B(z) =
\begin{cases}
G_B(z) & \text{if } |A| \text{ is even}, \\[4pt]
\hat{G}_B(z) & \text{if } |A| \text{ is odd}.
\end{cases}
\]
And the sign rule is distinguished by the diagonal and off diagonal cases.  For diagonal terms ($A = B$, i.e., $C = \emptyset$):
\[
\text{sign}(A, A) =
\begin{cases}
(-1)^{\frac{|A|}{2}} & \text{if } |A| \text{ is even}, \\[4pt]
(-1)^{\frac{|A|+1}{2}} & \text{if } |A| \text{ is odd}.
\end{cases}
\]
For off-diagonal terms ($A \neq B$, i.e. $C\neq \emptyset$):
\[
\text{sign}(A, B) = (-1)^{\sigma(A,B)}
\]
where $\sigma(A,B)$ is the Clifford sign exponent defined below.  The sign $(-1)^{\sigma(A,B)}$ arises from reducing the Clifford product $I_A I_B$ to standard form:
\[
I_A I_B = (-1)^{\sigma(A,B)} I_{A \triangle B}
\]
where $A \triangle B = (A \cup B) \setminus (A \cap B)$ is the symmetric difference, and
\[
\sigma(A, B) = |A \cap B| + \tau(A, B)
\]
with
\[
\tau(A, B) = \bigl|\{(a, b) \in A \times B : a > b\}\bigr|.
\]
The term $|A \cap B|$ accounts for signs from $I_i^2 = -1$, and $\tau(A,B)$ counts the transpositions needed to reorder the product.

\subsection{Setup and Goal}

We start with equations \eqref{e:sliceproduct_CSS} from Colombo--Sabadini--Struppa, \cite{MR2684426}*{Equation (9)}:
\begin{align}
(f_I * g_I)(z) &= \sum_{|A| \text{ even}} (-1)^{\frac{|A|}{2}} F_A(z) G_A(z) + \sum_{|A| \text{ odd}} (-1)^{\frac{|A|+1}{2}} F_A(z) \hat{G}_A(z) \notag \\
&\quad + \sum_{\substack{|A| \text{ even} \\ B \neq A}} F_A(z) G_B(z) I_A I_B + \sum_{\substack{|A| \text{ odd} \\ B \neq A}} F_A(z) \hat{G}_B(z) I_A I_B
\end{align}
where $\hat{G}(z) = \overline{G(\bar{z})}$, and $A, B$ range over subsets of $\{2, \ldots, n\}$.  We collect some preparatory lemmas and computations.

\subsubsection{Preliminary Lemmas}

\begin{lemma}[Clifford Product Reduction]
\label{lem:clifford}
For subsets $A, B \subseteq \{2, \ldots, n\}$ with elements written in increasing order, we have
\[
I_A I_B = (-1)^{\sigma(A,B)} I_{A \triangle B}
\]
where $A \triangle B = (A \cup B) \setminus (A \cap B)$ is the symmetric difference and
\[
\sigma(A, B) = |A \cap B| + \tau(A, B), \qquad \tau(A,B) = |\{(a,b) \in A \times B : a > b\}|.
\]
\end{lemma}

\begin{proof}
Starting from $I_A I_B = I_{a_1} \cdots I_{a_p} I_{b_1} \cdots I_{b_q}$ where $A = \{a_1 < \cdots < a_p\}$ and $B = \{b_1 < \cdots < b_q\}$, we need to:
\begin{enumerate}
    \item Reorder to bring matching elements together, and
    \item Cancel pairs using $I_i^2 = -1$.
\end{enumerate}

In the first step, we need to count the transpositions for reordering.  To move each $I_{b_j}$ to its correct position, we must swap it past all $I_{a_i}$ with $a_i > b_j$. Each swap contributes a factor of $-1$ since $I_{a_i} I_{b_j} = -I_{b_j} I_{a_i}$ for $a_i \neq b_j$.  The total number of such swaps is
\[
\tau(A, B) = \sum_{b \in B} |\{a \in A : a > b\}| = |\{(a,b) \in A \times B : a > b\}|.
\]

The second step cancels common elements.  Canceling common elements.  For each $i \in A \cap B$, after reordering we have $I_i I_i = -1$. This contributes a factor of $(-1)^{|A \cap B|}$.

Finally, we need to determine the remaining elements.  After cancellation, the remaining elements are exactly those in $A \triangle B$, and they appear in increasing order.

Putting these steps together therefore yields,
\[
I_A I_B = (-1)^{\tau(A,B)}  (-1)^{|A \cap B|}  I_{A \triangle B} = (-1)^{\sigma(A,B)} I_{A \triangle B}.
\]
\end{proof}

\begin{lemma}[Conjugation Rule]
\label{lem:conjugation}
For $I_A$ with $A \subseteq \{2, \ldots, n\}$ and a holomorphic function $G(z)$ on $\mathbb{C}_I$,
\[
I_A * G(z) =
\begin{cases}
G(z)  I_A & \text{if } |A| \text{ is even}, \\
\hat{G}(z) I_A & \text{if } |A| \text{ is odd},
\end{cases}
\]
where $\hat{G(z)} = \overline{G(\bar{z})}$. Moreover,
\[
I_A \cdot G(z) =
\begin{cases}
G(z)  I_A & \text{if } |A| \text{ is even}, \\
\overline{G}(z) I_A & \text{if } |A| \text{ is odd}.
\end{cases}
\]
\end{lemma}

\begin{proof}
The first assertion follows from the definition of $*$-product of the constant function $a(x)=I_A$ with $G$. The second assertion can be proved by determining
whether $I_A$ commutes or anticommutes with $I_1$.
\\
For a single generator $I_k$ with $k \geq 2$:
\[
I_k I_1 = -I_1 I_k
\]
so $I_k$ anticommutes with $I_1$.  For a product $I_A = I_{a_1} \cdots I_{a_m}$ where $|A| = m$:
\[
I_A I_1 = I_{a_1} \cdots I_{a_m} I_1 = (-1)^m I_1 I_{a_1} \cdots I_{a_m} = (-1)^{|A|} I_1 I_A.
\]
Now, for $w = u + vI_1$ on $\mathbb{C}_I$, with $u,v\in\mathbb{R}$:
\[
I_A \cdot w = I_A(u + vI_1) = u I_A +  I_A vI_1  = u I_A + (-1)^{|A|} vI_1 I_A=(u+(-1)^{|A|}vI_1)I_A.
\]
This implies:
\begin{itemize}
\item If $|A|$ is even: $I_A\cdot w = (u +  vI_1) I_A = w \cdot I_A$.
\item If $|A|$ is odd: $I_A \cdot w = (u -  vI_1) I_A = \overline{w} \cdot I_A$.
\end{itemize}
Since holomorphic functions are built from powers of $z$ with complex coefficients, we have for any holomorphic $G$:
\begin{itemize}
    \item $|A|$ even: $I_A\cdot G(z) = G(z)\cdot I_A$
    \item $|A|$ odd: $I_A\cdot G(z) = \overline{G({z})}\cdot I_A$.
\end{itemize}
This follows since we apply the observation about commuting $w$ and $I_A$ first to the coefficient and then to the power of $z$.
\end{proof}

\begin{lemma}
\label{l:trivialcounting}
For any non-negative integer $k$,
\[
(-1)^{\frac{k(k+1)}{2}} =
\begin{cases}
(-1)^{\frac{k}{2}} & \text{if } k \text{ is even}, \\
(-1)^{\frac{(k+1)}{2}} & \text{if } k \text{ is odd}.
\end{cases}
\]
\end{lemma}

\begin{proof}
This is done by case analysis by looking at $k$ even and odd.

Case 1: $k$ even. Write $k = 2m$. Then
\[
\frac{k(k+1)}{2} = \frac{2m(2m+1)}{2} = m(2m+1)=2m^2+m.
\]
So $(-1)^{m(2m+1)} = (-1)^{2m^2}(-1)^m=(-1)^m = (-1)^{\frac{k}{2}}$.

Case 2: $k$ odd. Write $k = 2m+1$. Then
\[
\frac{k(k+1)}{2} = \frac{(2m+1)(2m+2)}{2} = (2m+1)(m+1)=2m(m+1)+(m+1).
\]
So $(-1)^{(2m+1)(m+1)} = (-1)^{2m(m+1)+(m+1)} = (-1)^{2m(m+1)}(-1)^{m+1}= (-1)^{\frac{(k+1)}{2}}$.
\end{proof}

\subsection{Main Derivation}

With these lemmas, we show the formulas are equal.
\begin{prop}
\label{prop:productformula}
Equation \eqref{e:sliceproduct_CSS}, equivalently \eqref{e:sliceproduct_CSS2}, is equivalent to
\[
(f_I * g_I)(z) = \sum_{C \subseteq \{2,\ldots,n\}} H_C(z) \, I_C
\]
where
\[
H_C(z) := \sum_{\substack{A, B \\ A \triangle B = C}} \mathrm{sign}(A,B) \cdot F_A(z) \cdot \tilde{G}_B(z)
\]
with
\[
\tilde{G}_B =
\begin{cases}
G_B & \text{\rm if } |A| \text{\rm\ is even}, \\
\hat{G}_B & \text{\rm if } |A| \text{\rm\ is odd},
\end{cases}
\]
and
\[
\mathrm{sign}(A, B) =
\begin{cases}
(-1)^{\frac{|A|}{2}} & \text{\rm if } A = B \text{\rm\ and } |A| \text{\rm\ even}, \\
(-1)^{\frac{(|A|+1)}{2}} & \text{\rm if } A = B \text{\rm\  and } |A| \text{\rm\ odd}, \\
(-1)^{\sigma(A,B)} & \text{\rm if } A \neq B.
\end{cases}
\]
\end{prop}

One subtlety deserves attention: for diagonal terms ($A = B$), equation \eqref{e:sliceproduct_CSS} gives explicit signs $(-1)^{\frac{|A|}{2}}$ or $(-1)^{\frac{(|A|+1)}{2}}$, while for off-diagonal terms, the sign comes from the Clifford product reduction.

That these are consistent with checking what $(-1)^{\sigma(A,A)}$ gives:
\[
\sigma(A, A) = |A \cap A| + \tau(A, A) = |A| + \binom{|A|}{2} = |A| + \frac{|A|(|A|-1)}{2} = \frac{|A|(|A|+1)}{2}.
\]
So $(-1)^{\sigma(A,A)} = (-1)^{\frac{|A|(|A|+1)}{2}}$.  But, by Lemma \ref{l:trivialcounting} we can verify that
\[
(-1)^{\frac{|A|(|A|+1)}{2}} =
\begin{cases}
(-1)^{\frac{|A|}{2}} & \text{if } |A| \text{ even}, \\
(-1)^{\frac{(|A|+1)}{2}} & \text{if } |A| \text{ odd}.
\end{cases}
\]
This shows that the diagonal signs in equation \eqref{e:sliceproduct_CSS} are actually equal to $(-1)^{\sigma(A,A)}$, so we can write a unified formula:
\[
\mathrm{sign}(A, B) = (-1)^{\sigma(A,B)} \quad \text{for all } A, B.
\]

Using the sign convention just observed, we can restate an equivalent version of the decomposition.

\begin{prop}
\label{prop:Reformatted3.16}
Equation \eqref{e:sliceproduct_CSS}, equivalently \eqref{e:sliceproduct_CSS2}, is equivalent to
\[
(f_I * g_I)(z) = \sum_{C \subseteq \{2,\ldots,n\}} H_C(z) \, I_C
\]
where
\begin{equation}\label{functionHC}
H_C(z) = \sum_{\substack{A, B \subseteq \{2,\ldots,n\} \\ A \triangle B = C}} (-1)^{\sigma(A,B)} F_A(z) \tilde{G}_B(z)
\end{equation}
with
\[
\sigma(A, B) = |A \cap B| + |\{(a,b) \in A \times B : a > b\}|
\]
and
\[
\tilde{G}_B(z) :=
\begin{cases}
G_B(z) & \text{if } |A| \text{ is even}, \\
\hat{G}_B(z) = \overline{G_B(\bar{z})} & \text{if } |A| \text{ is odd}.
\end{cases}
\]
\end{prop}
This formula holds for \textbf{all} pairs $(A, B)$, including diagonal terms $A = B$.

\begin{proof}[Proof of Proposition \ref{prop:productformula} and Proposition \ref{prop:Reformatted3.16}]
We analyze each of the four sums in equation \eqref{e:sliceproduct_CSS}, equivalently \eqref{e:sliceproduct_CSS2} separately.  We recall the formula for \eqref{e:sliceproduct_CSS2} since it plays a role in what comes next:

\begin{align*}
(f_I * g_I)(z) &= \sum_{|A| \text{ even}} (-1)^{\frac{|A|}{2}} F_A(z) G_A(z) + \sum_{|A| \text{ odd}} (-1)^{\frac{|A|+1}{2}} F_A(z) \hat{G}_A(z) \notag \\
&\quad + \sum_{\substack{|A| \text{ even} \\ B \neq A}} F_A(z) G_B(z) I_A I_B + \sum_{\substack{|A| \text{ odd} \\ B \neq A}} F_A(z) \hat{G}_B(z) I_A I_B. \notag
\end{align*}

\textbf{Sum 1: Diagonal terms with $|A|$ even.}
\[
\sum_{|A| \text{ even}} (-1)^{\frac{|A|}{2}} F_A(z) G_A(z)
\]
These are scalar terms (no $I_A I_B$ factor since it equals $I_\emptyset = 1$ when $A = B$). They contribute to ${H}_\emptyset$ (i.e., $C=A\triangle B = \emptyset$) with:
\begin{itemize}
    \item $A = B$, so $A \triangle B = \emptyset$
    \item Sign: $(-1)^{\frac{|A|}{2}}$
    \item No hat since $|A|$ even.
\end{itemize}

\textbf{Sum 2: Diagonal terms with $|A|$ odd.}
\[
\sum_{|A| \text{ odd}} (-1)^{\frac{|A|+1}{2}} F_A(z) \hat{G}_A(z)
\]
Again, $A = B$ so these contribute to $H_\emptyset$ with:
\begin{itemize}
    \item $A = B$, so $A \triangle B = \emptyset$
    \item Sign: $(-1)^{\frac{(|A|+1)}{2}}$
    \item Hat appears since $|A|$ odd.
\end{itemize}

\textbf{Sum 3: Off-diagonal terms with $|A|$ even.}
\[
\sum_{\substack{|A| \text{ even} \\ B \neq A}} F_A(z) G_B(z) I_A I_B
\]
By Lemma \ref{lem:clifford}, $I_A I_B = (-1)^{\sigma(A,B)} I_{A \triangle B}$.  So this contributes to $H_C$ where $C = A \triangle B$ with:
\begin{itemize}
    \item Sign: $(-1)^{\sigma(A,B)}$ (from Clifford reduction)
    \item No hat since $|A|$ even
\end{itemize}

\textbf{Sum 4: Off-diagonal terms with $|A|$ odd.}
\[
\sum_{\substack{|A| \text{ odd} \\ B \neq A}} F_A(z) \hat{G}_B(z) I_A I_B
\]
Again by Lemma \ref{lem:clifford}, $I_A I_B = (-1)^{\sigma(A,B)} I_{A \triangle B}$.  This contributes to $H_C$ where $C = A \triangle B$ with:
\begin{itemize}
    \item Sign: $(-1)^{\sigma(A,B)}$ (from Clifford reduction)
    \item Hat appears since $|A|$ odd.
\end{itemize}

We conclude the proof by collecting terms.  Now we collect all contributions by the output basis element $I_C$:
\[
(f_I * g_I) (z) = \sum_{C \subseteq \{2,\ldots,n\}} \left( \sum_{\substack{A, B \\ A \triangle B = C}} \mathrm{sign}(A,B) \cdot F_A(z) \cdot \tilde{G}_B(z) \right) I_C = \sum_{C \subseteq \{2,\ldots,n\}} H_C(z) \, I_C
\]
where the sign and conjugation rules are exactly as stated.
\end{proof}

\subsection{Alternative Parametrization}

We present another way to describe the product.  The idea now is to describe all pairs $(A, B)$ of subsets of $\{2, \ldots, n\}$ such that $A \triangle B = C$ for a fixed $C$.

\begin{prop}
The pairs $(A, B)$ satisfying $A \triangle B = C$ are in bijection with pairs $(S, D)$ where $S \subseteq C$ and $D \subseteq C^c := \{2, \ldots, n\} \setminus C$. The bijection is given by:
\[
A = S \cup D, \qquad B = (C \setminus S) \cup D.
\]
\end{prop}

\begin{proof}
Recall that the symmetric difference is $A \triangle B = (A \cup B) \setminus (A \cap B)$, which consists of elements in exactly one of $A$ or $B$.

Step 1: Decomposing $A$ and $B$.  Given any pair $(A, B)$, we can partition the elements into three disjoint regions:
\begin{itemize}
    \item $A \cap B$ = elements in both $A$ and $B$
    \item $A \setminus B$ = elements in $A$ only
    \item $B \setminus A$ = elements in $B$ only
\end{itemize}
Step 2: Setting $C=A \triangle B$, then:
\begin{itemize}
    \item $(A \setminus B)$ and $(B \setminus A)$ are disjoint and partition $C$
    \item $A \cap B$ is disjoint from $C$, so $A \cap B \subseteq C^c$
\end{itemize}

Step 3: Let us define:
\begin{align*}
S &:= A \setminus B = A \cap C \quad \text{(the part of $C$ that comes from $A$)} \\
D &:= A \cap B \quad \text{(the common part)}.
\end{align*}

Then:
\begin{itemize}
    \item $S \subseteq C$ (since $A \setminus B \subseteq A \triangle B = C$)
    \item $D \subseteq C^c$ (since $A \cap B$ is disjoint from $A \triangle B = C$)
    \item $B \setminus A = C \setminus S$ (since $S$ and $B \setminus A$ partition $C$).
\end{itemize}

Step 4: Reconstructing $A$ and $B$ from $S$ and $D$.  From the definitions:
\begin{align*}
A &= (A \setminus B) \cup (A \cap B) = S \cup D \\
B &= (B \setminus A) \cup (A \cap B) = (C \setminus S) \cup D.
\end{align*}

Step 5: Verifying the bijection.  \emph{Forward direction:} Given $(A, B)$ with $A \triangle B = C$, we obtain $(S, D)$ with $S \subseteq C$ and $D \subseteq C^c$.

\emph{Reverse direction:} Given $(S, D)$ with $S \subseteq C$ and $D \subseteq C^c$, define $A = S \cup D$ and $B = (C \setminus S) \cup D$. We verify:
\begin{align*}
A \triangle B &= (A \setminus B) \cup (B \setminus A) \\
&= [(S \cup D) \setminus ((C \setminus S) \cup D)] \cup [((C \setminus S) \cup D) \setminus (S \cup D) ]\\
&= S \cup (C \setminus S) \quad \text{(since $D$ cancels out)} \\
&= C. \quad
\end{align*}
More explicitly:
\begin{itemize}
    \item $A \setminus B = (S \cup D) \setminus ((C \setminus S) \cup D) = S \setminus (C \setminus S) = S$ (since $S$ and $C \setminus S$ are disjoint)
    \item $B \setminus A = ((C \setminus S) \cup D) \setminus (S \cup D) = (C \setminus S) \setminus S = C \setminus S$
    \item Therefore $A \triangle B = S \cup (C \setminus S) = C$.
\end{itemize}
Also verify: $A \cap B = (S \cup D) \cap ((C \setminus S) \cup D) = D$ (since $S \cap (C \setminus S) = \emptyset$).

Step 6: Counting.  The number of pairs $(S, D)$ is $|2^C| \cdot |2^{C^c}| = 2^{|C|} \cdot 2^{n-1-|C|} = 2^{n-1}$.

This matches the number of pairs $(A, B)$ with $A \triangle B = C$: for each element $i \in C$, it must be in exactly one of $A$ or $B$ (2 choices), and for each element $j \in C^c$, it is either in both or neither (2 choices). Total: $2^{|C|} \cdot 2^{|C^c|} = 2^{n-1}$.
\end{proof}

Using this bijection, we can rewrite \eqref{functionHC} as:
\[
H_C(z) = \sum_{\substack{A, B \\ A \triangle B = C}} (-1)^{\sigma(A,B)} F_A(z) \tilde{G}_B(z)
= \sum_{D \subseteq C^c} \sum_{S \subseteq C} (-1)^{\sigma(S \cup D, \, (C \setminus S) \cup D)} F_{S \cup D}(z) \, \tilde{G}_{(C \setminus S) \cup D}(z)
\]
where:
\begin{itemize}
    \item $A = S \cup D$
    \item $B = (C \setminus S) \cup D$
    \item $|A| = |S| + |D|$ (since $S \subseteq C$ and $D \subseteq C^c$ are disjoint)
    \item $\tilde{G}_B = G_B$ if $|A| = |S| + |D|$ is even, and $\tilde{G}_B = \hat{G}_B$ if $|A|$ is odd
\end{itemize}

\section{Matrix formulation of the slice product and auxiliary equations}
\label{subsec:matrix-clifford-aux}

We now explain how the matrix formalism for the slice product yields the auxiliary B\'ezout equations in the Clifford setting, exactly as in the quaternionic case.

Let
$$
\mathcal I:=\mathcal P(\{2,\dots,n\}),
\qquad
|\mathcal I|=2^{n-1}.
$$
Write
$$
f_I(z)=\sum_{A\in\mathcal I}F_A(z)\,I_A,
\qquad
h_I(z)=\sum_{B\in\mathcal I}H_B(z)\,I_B.
$$
As seen in the last section we have that
$$
(f_I*h_I)(z)=\sum_{C\in\mathcal I}K_C(z)\,I_C,
$$
where
\begin{equation}
\label{eq:product-matrix-section}
K_C(z)=\sum_{\substack{A,B\in\mathcal I\\ A\triangle B=C}}
(-1)^{\sigma(A,B)}\,F_A(z)\,\widetilde H_B(z),
\end{equation}
with
\begin{equation}\label{eq:KC}
\widetilde H_B(z)=
\begin{cases}
H_B(z), & \text{if } |A| \text{ is even},\\
\hat{H}_B(z), & \text{if } |A| \text{ is odd}.
\end{cases}
\end{equation}
From this fact we can understand the product $f_I\ast h_I$ as a matrix multiplication.

Indeed, let $f$ be a slice monogenic function and let $\displaystyle f_I = \sum_{A\in\mathcal{I}} F_A I_A$ be its restriction to $\mathbb C_I$; then we can define the  matrix $M_f$ associated to $f_I$, which is given by
\begin{equation}
\label{eq:Mf-matrix-section}
(M_f)_{A,B}
=
(-1)^{\sigma(A\triangle B,B)}
\begin{cases}
F_{A\triangle B}, & \text{if } |A| \text{ is even},\\
\hat{F}_{A\triangle B}, & \text{if } |A| \text{ is odd}.
\end{cases}
\end{equation}
However, there is a different method to associate a matrix to the restriction $\ell_I(z)= \sum_{C\in\mathcal{I}} L_C(z)I_C$ of a slice monogenic function $\ell$. We define the matrix $\mathsf L$ whose first column vector  $\mathbf{l}$ has components described by:
\begin{equation}\label{columnvectors}
\mathbf{l}=\bigl(\mathbf{l}_C\bigr)_{C\in\mathcal I},
\qquad
\mathbf{l}_C=
\begin{cases}
L_C, & \text{if } |C| \text{ is even},\\
\hat{L}_C, & \text{if } |C| \text{ is odd}.
\end{cases}
\end{equation}
To determine the other column vectors of $\mathsf L$, indexed by $\mathcal I$, fix a multi-index $D\in\mathcal I$; the $D$-th column is obtained by considering $$\ell_I(z) * I_D=\ell_I(z)\cdot I_D=\sum_{C\in\mathcal{I}} L_C(z)I_C I_D=\sum_{C\in\mathcal{I}} (-1)^{\sigma(C,D)}L_C I_{C\triangle D}=\sum_{A\in\mathcal{I}} (-1)^{\sigma(A\triangle D,D)}L_{A\triangle D} I_{A}$$ where we put $A=C\triangle D$ and so $C=A\triangle D$.
Following the same construction as above (to get the first column of $\mathsf L$), i.e. formula \eqref{columnvectors}) we get that the $A$-entry of this column is either 
\begin{equation}\label{Leven}
(-1)^{\sigma(A\triangle D,D)}L_{A\triangle D}, \qquad |A| \text{\ even}
\end{equation}
or \begin{equation}\label{Lodd}
(-1)^{\sigma(A\triangle D,D)}\hat{L}_{A\triangle D}, \qquad |A| \text{\ odd}
\end{equation}
We note that this construction is fully equivalent, to the description via $M_f$, in fact the above $(A,D)$-entry of $\mathsf L$ coincides with the entry prescribed by \eqref{eq:Mf-matrix-section}. The first description
is more convenient for the function written on the left in the product $f_I*h_I$ while the second construction is more suitable to deal with $h_I$.
In particular, we follow this second procedure to define the two matrices associated with the functions $h_I(z)=\sum_{B\in\mathcal{I}} H_B(z)I_B$, $k_I(z)=\sum_{B\in\mathcal{I}} K_B(z)I_B$, where $(f_I*h_I)(z)=k_I(z)$.

With the notations just introduced, we have:
\begin{prop}\label{prop:matrix_equiv}
The equation $f_I * h_I = 1$ (the $N=1$ Corona equation) is equivalent to the matrix equation
\[
M_f \cdot M_h = \textnormal{Id}_{2^{n-1}}.
\]
\end{prop}

\begin{proof}
It suffices to verify that $M_f \mathbf{h} = \mathbf{k}$ where $\mathbf{h}$ and $\mathbf{k}$
are the column vectors of $h$ and $k = f * h$ defined in \eqref{columnvectors}.
We begin by considering the first column vectors and we fix a row index $A$. The $A$-th entry of $M_f \mathbf{h}$ is
\[
(M_f \mathbf{h})_A = \sum_{B \subseteq \{2,\ldots,n\}} (M_f)_{A,B} \cdot (\mathbf{h})_B .
\]
Substituting the definition of $(M_f)_{A,B}$ and $(\mathbf{h})_B$ gives
\[
(M_f \mathbf{h})_A
=
\sum_B
(-1)^{\sigma(A \triangle B,B)}
\begin{cases}
F_{A\triangle B}, & |A| \text{ even},\\
\hat{F}_{A\triangle B}, & |A| \text{ odd}
\end{cases}
\cdot
\begin{cases}
H_B, & |B| \text{ even},\\
\hat{H}_B, & |B| \text{ odd}.
\end{cases}
\]

Setting $A' = A \triangle B$, the sign
$(-1)^{\sigma(A',B)}$ is exactly the Clifford sign appearing in the product
\[
I_{A'} I_B = (-1)^{\sigma(A',B)} I_A .
\]
Thus the sum reproduces the coefficient formula for $K_A$ in the slice product
\eqref{eq:product-matrix-section}. Consequently,
\[
M_f \mathbf{h} = \mathbf{k}.
\]
Let us now consider the $D$-th column of $M_h$ whose entries are given in  \eqref{Leven}, \eqref{Lodd}. 
Fix $A,D\in\mathcal I$ and compute
\[
(M_fM_h)_{A,D}=\sum_{C\in\mathcal I}(M_f)_{A,C}\,(M_h)_{C,D} .
\]
We set $A'=A\triangle C$ and $B'=C\triangle D$ and we use the definition of the two matrices  which gives
\[
(M_f)_{A,C}=(-1)^{\sigma(A',C)}
\begin{cases}F_{A'},&|A|\ \text{even}\\ \hat F_{A'},&|A|\ \text{odd}\end{cases}
\qquad
(M_h)_{C,D}=(-1)^{\sigma(B',D)}
\begin{cases}H_{B'},&|C|\ \text{even}\\ \hat H_{B'},&|C|\ \text{odd} .\end{cases}
\]
Note that $C\mapsto(A',B')$ is a bijection from $\mathcal I$ onto the set of pairs
$(A',B')\in\mathcal I\times\mathcal I$ with $A'\triangle B'=E$ where $E:=A\triangle D$,  (its inverse being
$(A',B')\mapsto C=A\triangle A'$), and that
\begin{equation}\label{eq:parityD}
|C|=|A\triangle A'|\equiv|A|+|A'|\pmod 2 .
\end{equation}
\\
\emph{Case $|A|$ even.} By \eqref{eq:parityD}, $|C|\equiv|A'|$, so the second factor is
$H_{B'}$ if $|A'|$ is even and $\hat H_{B'}$ if $|A'|$ is odd, i.e. it is exactly
$\widetilde H_{B'}$ in the sense of \eqref{eq:KC}. Then we have
\[
\begin{split}
(M_fM_h)_{A,D}
&=\sum_{\substack{A',B'\in\mathcal I\\ A'\triangle B'=E}}
(-1)^{\sigma(A',C)+\sigma(B',D)}F_{A'}\widetilde H_{B'}\\
&=(-1)^{\sigma(E,D)}\!\!\sum_{\substack{A'\triangle B'=E}}\!\!
(-1)^{\sigma(A',B')}F_{A'}\widetilde H_{B'}\\
&=(-1)^{\sigma(E,D)}K_E ,
\end{split}
\]
by \eqref{eq:KC}. Since $|A|$ is even, $(-1)^{\sigma(E,D)}K_E=(M_{f\ast h})_{A,D}$. Here we used $(-1)^{\sigma (A',C)+\sigma(B',D)}=(-1)^{\sigma (A',B')+\sigma(E,D)}$ which follows by computing $I_{A'}I_{B'}I_D$.
\\
\emph{Case $|A|$ odd.} Now $|C|\equiv|A'|+1$, so the second factor is $\hat H_{B'}$ if
$|A'|$ is even and $H_{B'}$ if $|A'|$ is odd; in both cases it equals
$\widehat{\widetilde H_{B'}}$, because $\widehat{\cdot}$ is involutive. Since also
$(M_f)_{A,C}$ involves $\hat F_{A'}$, and since it is immediate that $\widehat{\cdot}$ is multiplicative and
additive, the same computation gives
\[
\begin{split}
(M_fM_h)_{A,D}
&=(-1)^{\sigma(E,D)}\sum_{A'\triangle B'=E}(-1)^{\sigma(A',B')}
\hat F_{A'}\,\widehat{\widetilde H_{B'}}\\
&=(-1)^{\sigma(E,D)}\,\widehat{\sum_{A'\triangle B'=E}(-1)^{\sigma(A',B')}F_{A'}\widetilde H_{B'}}\\
&=(-1)^{\sigma(E,D)}\hat K_E ,
\end{split}
\]
(the real signs $(-1)^{\sigma}$ are not affected by $\widehat{\cdot}\, $), which is
$(M_{f\ast h})_{A,D}$ because $|A|$ is odd. This proves $M_fM_h=M_{f\ast h}$. In particular, the equation $f*h = 1$ corresponds to
$M_{f*h}=M_1 = {\rm Id}_{2^{n-1}}$, which is equivalent to
\[
M_f M_h = \textnormal{Id}_{2^{n-1}}.
\]
\end{proof}

We note that the equality $f_I*h_I = 1$ is equivalent to $f*h = 1$ by the extension theorem (recall that, by definition, $f*h = {\rm ext}(f_I*h_I)$.  The same point of view extends immediately to several generators. Indeed the map $f\to M_f$ is trivially injective, unital and additive i.e. $(f+h)\to M_{f}+M_h$. Thus, from
$$
f_1*h_1+\cdots+f_N*h_N=1,
$$ we deduce, by taking the restriction to $\mathbb C_I$, the equality
$$
(f_1*h_1)_I+\cdots+(f_N*h_N)_I=f_{1,I}*h_{1,I}+\cdots +f_{N,I}*h_{N,I}=1,
$$
and the matrix formulation becomes 
\begin{equation}
\label{eq:Corona-matrix-many-generators}
M_{f_1}M_{h_1}+\cdots+M_{f_N}M_{h_N}=\textnormal{Id}_{2^{n-1}}.
\end{equation}
More precisely, it follows multiplying $\mathsf{F}=[M_{f_1}\ldots M_{f_N}]$ and the column block matrix $\mathsf{H}=[M_{h_1};\ldots ;M_{h_N}]$ and 
$$
M_{f_1*h_1+\cdots+f_N*h_N}=\mathsf{F}\mathsf{H}.$$
Accordingly, the auxiliary B\'ezout equations in the multi-generator setting should arise from the determinant of the matrix on the left-hand side of \eqref{eq:Corona-matrix-many-generators}.

\section{\texorpdfstring{The Clifford Algebra Case $\mathbb{R}_3$}{The Clifford Algebra Case R3}} \label{Sect5}

For $j=1,\dots,N$ write the restriction $f_{j,I}$ to the complex plane $\mathbb C_I$ as
\[
f_{j,I}
=
F_{1}^{(j)}
+
F_{2}^{(j)} I_2
+
F_{3}^{(j)} I_3
+
F_{23}^{(j)} I_{23},
\qquad
F_{\alpha}^{(j)}\in H^\infty(\DD).
\]

Define the $4\times4$ matrix associated with $f_{j,I}$ by
\begin{equation}\label{eq:Mf_clifford_general}
M_{f_j}(z)=
\begin{pmatrix}
F_{1}^{(j)}(z) & -F_{2}^{(j)}(z) & -F_{3}^{(j)}(z) & -F_{23}^{(j)}(z) \\
\hat{F}_{2}^{(j)}(z) & \hat{F}_{1}^{(j)}(z) & -\hat{F}_{23}^{(j)}(z) & \hat{F}_{3}^{(j)}(z) \\
\hat{F}_{3}^{(j)}(z) & \hat{F}_{23}^{(j)}(z) & \hat{F}_{1}^{(j)}(z) & -\hat{F}_{2}^{(j)}(z) \\
F_{23}^{(j)}(z) & -F_{3}^{(j)}(z) & F_{2}^{(j)}(z) & F_{1}^{(j)}(z)
\end{pmatrix}.
\end{equation}
Similarly, if
\[
h_{j,I}
=
H_{1}^{(j)}
+
H_{2}^{(j)} I_2
+
H_{3}^{(j)} I_3
+
H_{23}^{(j)} I_{23},
\]
define
\begin{equation}\label{eq:Mg_clifford_general}
M_{h_j}(z)=
\begin{pmatrix}
H_{1}^{(j)}(z) & -H_{2}^{(j)}(z) & -H_{3}^{(j)}(z) & -H_{23}^{(j)}(z) \\
\hat{H}_{2}^{(j)}(z) & \hat{H}_{1}^{(j)}(z) & -\hat{H}_{23}^{(j)}(z) & \hat{H}_{3}^{(j)}(z) \\
\hat{H}_{3}^{(j)}(z) & \hat{H}_{23}^{(j)}(z) & \hat{H}_{1}^{(j)}(z) & -\hat{H}_{2}^{(j)}(z) \\
H_{23}^{(j)}(z) & -H_{3}^{(j)}(z) & H_{2}^{(j)}(z) & H_{1}^{(j)}(z)
\end{pmatrix}.
\end{equation}
These matrices encode left $*$-multiplication.  In particular we have \begin{equation}\label{eq:matrix_product_clifford}
M_{f_j}(z)M_{h_j}(z)=M_{\,f_j*h_j}(z).
\end{equation}

\medskip

We now combine the generators into block matrices.  Define
\begin{equation}\label{eq:blockF_clifford}
\mathsf{F}(z)
:=
\bigl(
M_{f_1}(z)\ \cdots\ M_{f_N}(z)
\bigr)
\in H^\infty(\DD)^{4\times 4N},
\end{equation}
and
\begin{equation}\label{eq:blockH_clifford}
\mathsf{H}(z)
:=
\begin{pmatrix}
M_{h_1}(z)\\
\vdots\\
M_{h_N}(z)
\end{pmatrix}
\in H^\infty(\DD)^{4N\times 4}.
\end{equation}
A direct computation using \eqref{eq:matrix_product_clifford} shows
\begin{equation}\label{eq:block_identity_clifford}
\mathsf{F}(z) \mathsf{H}(z)
=
M_{f_1}(z)M_{h_1}(z)
+\cdots+
M_{f_N}(z)M_{h_N}(z).
\end{equation}
Consequently,
\begin{equation*}
f_1*h_1(q)+\cdots+f_N*h_N(q)=1,\ \forall q\in\mathbb B
\quad\Longleftrightarrow\quad f_{1,I}*h_{1,I}(z)+\cdots+f_{N,I}*h_{N,I}(z)=1,\ \forall z\in\mathbb D
\end{equation*}
\begin{equation}\label{eq:clifford_bezout_matrix}
\quad\Longleftrightarrow\quad
\mathsf{F}(z) \mathsf{H}(z)=\textnormal{Id}_4 \quad\Longleftrightarrow\quad \mathsf{H}^{\top}(z) \mathsf{F}^{\top}(z)=\textnormal{Id}_4.
\end{equation}

Thus the Clifford B\'ezout equation can be expressed as a single matrix identity involving the block matrices $\mathsf{F}$ and $\mathsf{H}$.

Alternatively, by looking at the entries of the matrix, and taking for example the first column of the product, we can expand these as a system of equations:
\begin{align}
\sum_{j=1}^{N} \Big[ F_1^{(j)} H_1^{(j)} - F_2^{(j)} \hat H_2^{(j)} - F_3^{(j)} \hat H_3^{(j)} - F_{23}^{(j)} H_{23}^{(j)} \Big] &= 1 \\
\sum_{j=1}^{N} \Big[ F_1^{(j)} H_2^{(j)} + F_2^{(j)} \hat H_1^{(j)} + F_3^{(j)} \hat H_{23}^{(j)} - F_{23}^{(j)} H_3^{(j)} \Big] &= 0 \\
\sum_{j=1}^{N} \Big[ F_1^{(j)} H_3^{(j)} - F_2^{(j)} \hat H_{23}^{(j)} + F_3^{(j)} \hat H_1^{(j)} + F_{23}^{(j)} H_2^{(j)} \Big] &= 0 \\
\sum_{j=1}^{N} \Big[ F_1^{(j)} H_{23}^{(j)} + F_2^{(j)} \hat H_3^{(j)} - F_3^{(j)} \hat H_2^{(j)} + F_{23}^{(j)} H_1^{(j)} \Big] &= 0.
\end{align}

\subsection{Solving the Clifford B\'ezout equation via the matrix Corona theorem}

Assume that the data $\{f_j\}_{j=1}^N$ satisfy a Corona lower bound strong enough so that the matrix Corona theorem applies to the block matrix $\mathsf{F}$ defined in \eqref{eq:blockF_clifford}.  Then there exists
\[
\mathsf{P}(z)\in H^\infty(\DD)^{4N\times 4}
\]
such that
\begin{equation}\label{eq:right_inverse_P}
\mathsf{F}(z) \mathsf{P}(z) = \textnormal{Id}_4\quad\Longleftrightarrow\quad \mathsf{P}^{\top}(z) \mathsf{F}^{\top}(z) = \textnormal{Id}_4.
\end{equation}

The matrix $\mathsf{P}$ obtained from the Corona theorem does not necessarily have the block structure \eqref{eq:blockH_clifford}.  To recover this structure we use the Clifford symmetry matrices corresponding to the generators $I_2$ and $I_3$.

Define the constant $4\times4$ matrices
\begin{equation}\label{eq:J4_clifford}
J_4^{(2)}=
\begin{pmatrix}
0 & 1 & 0 & 0\\
-1 & 0 & 0 & 0\\
0 & 0 & 0 & -1\\
0 & 0 & 1 & 0
\end{pmatrix},
\qquad
J_4^{(3)}=
\begin{pmatrix}
0 & 0 & 1 & 0\\
0 & 0 & 0 & 1\\
-1 & 0 & 0 & 0\\
0 & -1 & 0 & 0
\end{pmatrix}.
\end{equation}
These are the matrices that represent right multiplication by the imaginary units $I_2$ and $I_3$ for elements $x\in\mathbb{R}_3$.  Note that these matrices anti-commute, $J_4^{(2)}J_4^{(3)}=-J_4^{(3)}J_4^{(2)}$.

For the block matrices define
\[
J_{4N}^{(2)}:=\operatorname{diag}(J_4^{(2)},\dots,J_4^{(2)}),
\qquad
J_{4N}^{(3)}:=\operatorname{diag}(J_4^{(3)},\dots,J_4^{(3)}).
\]
A direct computation shows that each matrix $M_{f_j}$ satisfies
\begin{equation}\label{eq:symmetry_Mf}
\widehat{M_{f_j}}
=
J_4^{(2)} M_{f_j} (J_4^{(2)})^{-1}
=
J_4^{(3)} M_{f_j} (J_4^{(3)})^{-1}.
\end{equation}
Consequently the block matrix $\mathsf{F}$ satisfies
\begin{equation}\label{eq:symmetry_blockF}
\widehat{\mathsf{F}}
=
J_4^{(2)}  \mathsf{F} (J_{4N}^{(2)})^{-1}
=
J_4^{(3)} \mathsf{F} (J_{4N}^{(3)})^{-1}.
\end{equation}

Using \eqref{eq:right_inverse_P} and \eqref{eq:symmetry_blockF} we obtain that the matrices
\[
\mathcal T_2(\mathsf{P}):=(J_{4N}^{(2)})^{-1}\widehat{\mathsf{P}}\,J_4^{(2)},
\qquad
\mathcal T_3(\mathsf{P}):=(J_{4N}^{(3)})^{-1}\widehat{\mathsf{P}}\,J_4^{(3)}
\]
also satisfy
\[
\mathsf{F}\,\mathcal T_2(\mathsf{P})=\textnormal{Id}_4,
\qquad
\mathsf{F}\,\mathcal T_3(\mathsf{P})=\textnormal{Id}_4 .
\]
We now symmetrize in an analogous fashion, as in the quaternionic case.  Define symmetrizations relative to the imaginary units:
\begin{align*}
\mathsf{P}^{(2)}:=\frac12\bigl(\mathsf{P}+\mathcal T_2(\mathsf{P})\bigr),\\
\mathsf{P}^{(3)}:=\frac12\bigl(\mathsf{P}+\mathcal T_3(\mathsf{P})\bigr),\\
\end{align*}
Then define:
\begin{align*}
\mathsf{H} & :=\left(\mathsf{P}^{(2)}\right)^{(3)} =\frac12\bigl(\mathsf{P}^{(2)}+\mathcal T_3(\mathsf{P}^{(2)})\bigr)\\
& =\frac{1}{4}\left(\mathsf{P}+\mathcal{T}_2(\mathsf{P})+\mathcal{T}_3(\mathsf{P})+
\mathcal{T}_3\mathcal{T}_2(\mathsf{P})\right)\\
& =\frac{1}{4}\left(\mathsf{P}+\mathcal{T}_2(\mathsf{P})+\mathcal{T}_3(\mathsf{P})+\mathcal{T}_2\mathcal{T}_3(\mathsf{P})\right)\\
& =\frac{1}{4}\left(\mathsf{P}+(J_{4N}^{(2)})^{-1}\widehat{\mathsf{P}}\,J_4^{(2)}+(J_{4N}^{(3)})^{-1}\widehat{\mathsf{P}}\,J_4^{(3)}+(J_{4N}^{(3)})^{-1}(J_{4N}^{(2)})^{-1} \mathsf{P}\,J_4^{(2)}\,J_4^{(3)}\right)\\
& =\frac{1}{4}\left(\mathsf{P}+(J_{4N}^{(2)})^{-1}\widehat{\mathsf{P}}\,J_4^{(2)}+(J_{4N}^{(3)})^{-1}
\widehat{\mathsf{P}}\,J_4^{(3)}+(J_{4N}^{(2)})^{-1}(J_{4N}^{(3)})^{-1} \mathsf{P}\,J_4^{(3)}\,J_4^{(2)}\right).
\end{align*}
Then $\mathsf{H}$ still satisfies
\begin{equation}\label{eq:block_solution_identity}
\mathsf{F}(z) \mathsf{H}(z)=\textnormal{Id}_4 .
\end{equation}

Moreover the symmetry conditions imply that $ \mathsf{H}$ has the block form
\[
\mathsf{H}=
\begin{pmatrix}
M_{h_1}\\
\vdots\\
M_{h_N}
\end{pmatrix},
\]
where each block $M_{h_j}$ has the structure \eqref{eq:Mg_clifford_general} for some functions
$H_{1}^{(j)},H_{2}^{(j)},H_{3}^{(j)},H_{23}^{(j)}\in H^\infty(\DD)$.  Substituting this form into \eqref{eq:block_solution_identity} yields
\[
M_{f_1}(z)M_{h_1}(z)+\cdots+M_{f_N}(z)M_{h_N}(z)=\textnormal{Id}_4,
\]
which is equivalent to
\[
f_{I,1}(z)*h_{I,1}(z)+\cdots+f_{I,N}(z)*h_{I,N}(z)=1, \qquad z\in \mathbb{B}\cap\mathbb{C}_I\cong \mathbb{D}.
\]
By the extension theorem (Theorem \ref{thm:extension}), setting $h_j={\rm ext}(h_{j,I})$ we get the formula

\[
f_{1}*h_{1}+\cdots+f_{N}*h_{N}=1.
\]
The bounds follows from the fact that, using the Representation Formula, we have
$$
\left\Vert h\right\Vert_{H^{\infty}(\mathbb{B})}=\sup_{q\in\mathbb{B}}\vert h(q)\vert\lesssim\left\Vert h_I\right\Vert_{H^{\infty}(\mathbb{B}\cap\mathbb{C}_I)}.
$$
Thus the Clifford B\'ezout equation admits bounded holomorphic solutions obtained from the matrix Corona theorem after projecting the Corona solution onto the Clifford-symmetric subspace.

\begin{rem}[Underlying matrix structure]\label{rem:matrix_structure}
A key structural feature shared by the quaternionic and Clifford slice settings is that the corresponding function spaces admit a decomposition into finitely many complex holomorphic components which are stable under the anti-linear involution
\[
f \mapsto \hat f, \qquad \hat f(z)=\overline{f(\bar z)}.
\]

In the quaternionic case the restriction to $\mathbb C_I$ of a function $f\in H^\infty(\mathbb{B})$ can be written in the form
\[
f_I(z)=F(z)+G(z)J,
\qquad F,G\in H^\infty(\DD),
\]
while in the Clifford setting $\mathbb{R}_3$ one has
\[
f_I(z)=F_1(z)+F_2(z)I_2+F_3(z)I_3+F_{23}(z)I_{23},
\qquad F_\alpha\in H^\infty(\DD).
\]

Because the slice product interacts with the involution $f\mapsto\hat f$ in a linear way, the multiplication of such functions becomes linear in the coefficient functions.  Consequently each function $f$ can be encoded by a finite matrix with entries in $H^\infty(\DD)$ so that
\[
M_{f*h}=M_f\,M_h .
\]

Thus the B\'ezout equation in the slice or Clifford setting reduces to a matrix equation for holomorphic functions, allowing the use of the matrix Corona theorem. By extending the solutions (via the Extension Theorem which gives unique functions), we obtain the result on $\mathbb B$.
\end{rem}

\section{Solving the general Clifford B\'ezout equation via the matrix Corona theorem}
\label{subsec:Bezout-general}

We now extend the procedure of the previous subsection to the general case of slice monogenic functions with values in $\mathbb{R}_n$ for any $n \ge 2$. The construction parallels the $\mathbb{R}_3$ case but uses $n-1$ commuting involutions, one for each Clifford generator.

\subsection{Setup}

Let $f_{1}, \ldots, f_{N}$ be slice monogenic functions on the unit ball $\mathbb B$ with values in $\mathbb{R}_n$. Consider their restrictions $f_{I,1}, \ldots, f_{I,N}$ to a complex plane $\mathbb C_I$ so that we can associate the respective matrices $M_{f_j}$ of size $2^{n-1}\times 2^{n-1}$ defined by \eqref{eq:Mf-matrix-section}. Form the block matrix
\[
\mathsf{F}(z) := \bigl(M_{f_1}(z)\ \cdots\ M_{f_N}(z)\bigr) \in H^\infty(\mathbb{D})^{2^{n-1}\times N\cdot 2^{n-1}},
\]
where $\mathbb{D}=\mathbb{B}\cap\mathbb{C}_I$.
Assume that the data $\{f_{j}\}_{j=1}^N$, and so their restrictions, satisfy a Corona lower bound strong enough so that the matrix Corona theorem applies to $\mathsf{F}$, yielding a matrix
\[
\mathsf{P}(z) \in H^\infty(\mathbb{D})^{N\cdot 2^{n-1} \times 2^{n-1}}
\]
satisfying
\begin{equation}\label{eq:right_inverse_P_general}
\mathsf{F}(z) \mathsf{P}(z) = \mathrm{Id}_{2^{n-1}} \quad\Longleftrightarrow\quad \mathsf{P}^{\top}(z) \mathsf{F}^{\top}(z) = \mathrm{Id}_{2^{n-1}}.
\end{equation}
As in the lower-dimensional cases, $\mathsf{P}$ need not have the block structure required for the recovery of slice monogenic solutions; we use the Clifford symmetries to project it onto the correct subspace.

\subsection{The Clifford symmetry matrices}

For each $k \in \{2, 3, \ldots, n\}$, define the constant $2^{n-1} \times 2^{n-1}$ matrix
\begin{equation}\label{eq:J_general_def}
\bigl(J^{(k)}\bigr)_{A,B} := \begin{cases}
(-1)^{\sigma(A,\{k\})}, & B = A \triangle \{k\},\\
0, & \text{otherwise,}
\end{cases}
\qquad A, B \in \mathcal I,
\end{equation}
where $\sigma$ is the same sign function appearing in \eqref{eq:Mf-matrix-section}. Equivalently, $J^{(k)}$ represents right multiplication by the imaginary unit $I_k$ on $\mathbb{R}_n$ in the basis $\{I_A\}_{A \in \mathcal I}$.

\begin{lemma}\label{lem:J-properties}
The matrices $J^{(k)}$, $k \in \{2,\ldots,n\}$, satisfy:
\begin{enumerate}
\item $\bigl(J^{(k)}\bigr)^2 = -\mathrm{Id}_{2^{n-1}}$ for every $k$;
\item $J^{(k)} J^{(\ell)} = -J^{(\ell)} J^{(k)}$ for every $k \ne \ell$.
\end{enumerate}
\end{lemma}

\begin{proof}
These identities encode the relations $I_k^2 = -1$ and $I_k I_\ell = -I_\ell I_k$ ($k \ne \ell$) of the Clifford algebra $\mathbb{R}_n$, transported to matrix form via the representation \eqref{eq:J_general_def}.
\end{proof}

For each $k$, define the block-diagonal lift
\[
J_N^{(k)} := \operatorname{diag}\bigl(J^{(k)}, \ldots, J^{(k)}\bigr) \in \mathbb{R}^{N\cdot 2^{n-1}\times N\cdot 2^{n-1}}\qquad (N \text{ blocks}).
\]
Properties (1) and (2) of Lemma \ref{lem:J-properties} carry over to $J_N^{(k)}$ verbatim.

\subsubsection{Symmetry of $M_f$ and the involutions $\mathcal T_k$}

\begin{lemma}\label{lem:symmetry-Mf-general}
For each $k \in \{2, \ldots, n\}$ and each $j \in \{1, \ldots, N\}$,
\begin{equation}\label{eq:symmetry_Mf_general}
\widehat{M_{f_j}} = J^{(k)} M_{f_j} \bigl(J^{(k)}\bigr)^{-1},
\end{equation}
and consequently
\begin{equation}\label{eq:symmetry_blockF_general}
\widehat{\mathsf{F}} = J^{(k)} \mathsf{F} \bigl(J_N^{(k)}\bigr)^{-1}.
\end{equation}
\end{lemma}

\begin{proof}
Drop the subscript $j$ and write $K=\{k\}$. Since $\bigl(J^{(k)}\bigr)^{-1}=-J^{(k)}$, identity \eqref{eq:symmetry_Mf_general} is equivalent to
\begin{equation}\label{eq:symmetry_Mf_equiv}
J^{(k)}\,M_f \;=\; \widehat{M_f}\,J^{(k)},
\end{equation}
which we verify entry-by-entry.

\smallskip\noindent\emph{Sign lemma.}\;
We first record that $\sigma$ is bilinear modulo $2$ in the symmetric-difference structure: for all subsets $X,Y,Z\subseteq\{2,\ldots,n\}$,
\begin{equation}\label{eq:sigma-bilinear}
\sigma(X\triangle Y,\,Z)\equiv \sigma(X,Z)+\sigma(Y,Z),\qquad
\sigma(X,\,Y\triangle Z)\equiv \sigma(X,Y)+\sigma(X,Z)\pmod 2.
\end{equation}
Indeed, the intersection-cardinality term satisfies $|X\cap(Y\triangle Z)|=|X\cap Y|+|X\cap Z|-2|X\cap Y\cap Z|$, and the inversion-count term satisfies
\[
\tau(X,Y\triangle Z)=\tau(X,Y)+\tau(X,Z)-2\,\tau(X,Y\cap Z),
\]
both of which reduce to the claimed bilinearity mod $2$. The left identity is analogous.

\smallskip\noindent\emph{Comparing entries.}\;
Fix $A,B\subseteq\{2,\ldots,n\}$. Since $J^{(k)}$ has a single nonzero entry in each row, namely $\bigl(J^{(k)}\bigr)_{A,A\triangle K}=(-1)^{\sigma(A,K)}$,
\[
\bigl(J^{(k)}\,M_f\bigr)_{A,B}
=(-1)^{\sigma(A,K)}\,(M_f)_{A\triangle K,\,B}.
\]
Substituting the definition \eqref{eq:Mf-matrix-section} and using $(A\triangle K)\triangle B=(A\triangle B)\triangle K$,
\begin{equation}\label{eq:LHS-entry}
\bigl(J^{(k)}\,M_f\bigr)_{A,B}
=(-1)^{\sigma(A,K)+\sigma((A\triangle B)\triangle K,\,B)}\cdot
\begin{cases}
F_{(A\triangle B)\triangle K}, & |A\triangle K|\text{ even,}\\
\hat F_{(A\triangle B)\triangle K}, & |A\triangle K|\text{ odd.}
\end{cases}
\end{equation}
Similarly, $J^{(k)}$ has a single nonzero entry in each column, $\bigl(J^{(k)}\bigr)_{B\triangle K,\,B}=(-1)^{\sigma(B\triangle K,K)}$, so
\[
\bigl(\widehat{M_f}\,J^{(k)}\bigr)_{A,B}
=(\widehat{M_f})_{A,\,B\triangle K}\cdot(-1)^{\sigma(B\triangle K,\,K)}.
\]
Now $\widehat{M_f}$ is obtained from $M_f$ by interchanging ${F}_C\leftrightarrow\hat{{F}}_C$ in every entry (the matrix entries are scalar-valued, so the hat acts coefficientwise and the sign factors are real). Substituting \eqref{eq:Mf-matrix-section} for $(\widehat{M_f})_{A,\,B\triangle K}$ and using $A\triangle(B\triangle K)=(A\triangle B)\triangle K$,
\begin{equation}\label{eq:RHS-entry}
\bigl(\widehat{M_f}\,J^{(k)}\bigr)_{A,B}
=(-1)^{\sigma((A\triangle B)\triangle K,\,B\triangle K)+\sigma(B\triangle K,\,K)}\cdot
\begin{cases}
\hat{F}_{(A\triangle B)\triangle K}, & |A|\text{ even,}\\
F_{(A\triangle B)\triangle K}, & |A|\text{ odd.}
\end{cases}
\end{equation}

\smallskip\noindent\emph{Matching the function symbols.}\;
Since $|A\triangle K|=|A|\pm 1$, the parities of $|A|$ and $|A\triangle K|$ are opposite. The case ``$|A\triangle K|$ even'' in \eqref{eq:LHS-entry} thus coincides with ``$|A|$ odd'' in \eqref{eq:RHS-entry}, and both select $F_{(A\triangle B)\triangle K}$; the other case selects $\hat F_{(A\triangle B)\triangle K}$ on both sides.

\smallskip\noindent\emph{Matching the signs.}\;
It remains to show that
\begin{equation}\label{eq:sign-identity}
\sigma(A,K)+\sigma((A\triangle B)\triangle K,\,B)
\;\equiv\;
\sigma((A\triangle B)\triangle K,\,B\triangle K)+\sigma(B\triangle K,\,K)\pmod 2.
\end{equation}
Applying right-bilinearity \eqref{eq:sigma-bilinear} to the first term on the right with $Y=B$, $Z=K$ gives
\[
\sigma((A\triangle B)\triangle K,\,B\triangle K)\equiv\sigma((A\triangle B)\triangle K,\,B)+\sigma((A\triangle B)\triangle K,\,K)\pmod 2,
\]
and applying left-bilinearity to $\sigma((A\triangle B)\triangle K,\,K)$ yields
\[
\sigma((A\triangle B)\triangle K,\,K)\equiv\sigma(A,K)+\sigma(B,K)+\sigma(K,K)\pmod 2.
\]
Combining and using left-bilinearity on the second term on the right, $\sigma(B\triangle K,K)\equiv\sigma(B,K)+\sigma(K,K)$, every term other than $\sigma(A,K)+\sigma((A\triangle B)\triangle K,B)$ cancels in pairs modulo $2$, establishing \eqref{eq:sign-identity}.

\smallskip
Together, \eqref{eq:LHS-entry} and \eqref{eq:RHS-entry} are equal in every entry $(A,B)$, proving \eqref{eq:symmetry_Mf_equiv} and hence \eqref{eq:symmetry_Mf_general}. The block version \eqref{eq:symmetry_blockF_general} follows immediately since $J_N^{(k)}$ acts blockwise on the columns of~$F$.
\end{proof}

For each $k \in \{2, \ldots, n\}$, define the operator $\mathcal T_k$ on $H^\infty(\mathbb{D})^{N\cdot 2^{n-1}\times 2^{n-1}}$ by
\begin{equation}\label{eq:Tk_def}
\mathcal T_k(\mathsf{P}) := \bigl(J_N^{(k)}\bigr)^{-1}\, \widehat{\mathsf{P}}\, J^{(k)}.
\end{equation}

\begin{prop}\label{prop:Tk_properties}
The operators $\{\mathcal T_k\}_{k=2}^n$ satisfy:
\begin{enumerate}
\item Each $\mathcal T_k$ is an involution: $\mathcal T_k^2 = \mathrm{Id}$.
\item The operators commute pairwise: $\mathcal T_k \mathcal T_\ell = \mathcal T_\ell \mathcal T_k$ for all $k, \ell \in \{2, \ldots, n\}$.
\item If $\mathsf{P}$ satisfies $\mathsf{F} \mathsf{P} = \mathrm{Id}_{2^{n-1}}$, then $\mathcal T_k(\mathsf{P})$ does too:
\[
\mathsf{F} \mathcal T_k(\mathsf{P}) = \mathrm{Id}_{2^{n-1}} \qquad \text{for every } k.
\]
\end{enumerate}
\end{prop}

\begin{proof}
For (1), compute
\[
\mathcal T_k^2(\mathsf{P}) = \bigl(J_N^{(k)}\bigr)^{-1}\,\widehat{\bigl(J_N^{(k)}\bigr)^{-1}\,\widehat{\mathsf{P}}\, J^{(k)}}\,J^{(k)} = \bigl(J_N^{(k)}\bigr)^{-2}\, \mathsf{P}\,\bigl(J^{(k)}\bigr)^2.
\]
Since $J^{(k)}, J_N^{(k)}$ are real, $\widehat{J^{(k)}} = J^{(k)}$ and similarly for $J_N^{(k)}$. By Lemma \ref{lem:J-properties}, $(J^{(k)})^2 = -\mathrm{Id}$ and $(J_N^{(k)})^{-2} = -\mathrm{Id}$, so the two minus signs cancel: $\mathcal T_k^2(\mathsf{P}) = \mathsf{P}$.

For (2), let $k \ne \ell$. Then
\begin{align*}
\mathcal T_k \mathcal T_\ell(\mathsf{P}) &= \bigl(J_N^{(k)}\bigr)^{-1}\bigl(J_N^{(\ell)}\bigr)^{-1}\, \mathsf{P}\, J^{(\ell)}\, J^{(k)},\\
\mathcal T_\ell \mathcal T_k(\mathsf{P}) &= \bigl(J_N^{(\ell)}\bigr)^{-1}\bigl(J_N^{(k)}\bigr)^{-1}\, \mathsf{P}\, J^{(k)}\, J^{(\ell)}.
\end{align*}
By Lemma \ref{lem:J-properties}, $J^{(\ell)} J^{(k)} = -J^{(k)} J^{(\ell)}$ and likewise for $J_N$. Each side acquires two minus signs, one from the left and one from the right, which cancel. Hence $\mathcal T_k \mathcal T_\ell = \mathcal T_\ell \mathcal T_k$.

For (3), apply $\widehat\cdot$ to the identity $\mathsf{F} \mathsf{P} = \mathrm{Id}_{2^{n-1}}$:
\[
\widehat{\mathsf{F}} \widehat{\mathsf{P}} = \mathrm{Id}_{2^{n-1}}.
\]
Substitute \eqref{eq:symmetry_blockF_general} to get
\[
\bigl(J^{(k)} \mathsf{F} (J_N^{(k)})^{-1}\bigr) \widehat{\mathsf{P}} = \mathrm{Id}_{2^{n-1}}, \]
upon using $-J^{(k)} = \bigl(J^{(k)}\bigr)^{-1}$, since $J^{(k)}$ is real with $(J^{(k)})^2 = -\mathrm{Id}$, 
one then obtains $\mathsf{F}\mathcal T_k(\mathsf{P}) = \mathrm{Id}_{2^{n-1}}$.
\end{proof}

\subsection{The symmetrized solution}

In light of Proposition \ref{prop:Tk_properties}, the projection onto the joint fixed subspace of $\{\mathcal T_k\}_{k=2}^n$ is realized by iterated averaging. Define recursively
\[
\mathsf{P}^{(2)} := \tfrac12\bigl(\mathsf{P} + \mathcal T_2(\mathsf{P})\bigr),
\]
and, for $k = 3, 4, \ldots, n$,
\[
\mathsf{P}^{(k)} := \tfrac12\bigl(\mathsf{P}^{(k-1)} + \mathcal T_k(\mathsf{P}^{(k-1)})\bigr).
\]
Let $\mathsf{H} := \mathsf{P}^{(n)}$.

\begin{prop}\label{prop:G-formula}
The matrix $\mathsf{H}$ admits the closed form
\begin{equation}\label{eq:G_avg}
\mathsf{H} \;=\; \frac{1}{2^{n-1}} \sum_{S \subseteq \{2, 3, \ldots, n\}} \mathcal T_S(\mathsf{P}),
\end{equation}
where $\mathcal T_S := \prod_{k \in S} \mathcal T_k$ (well-defined since the $\mathcal T_k$'s commute, with $\mathcal T_\emptyset = \mathrm{Id}_{2^{n-1}}$). Moreover, $\mathsf{H}$ is the projection of $\mathsf{P}$ onto the joint fixed subspace of $\{\mathcal T_k\}_{k=2}^n$, and
\begin{equation}\label{eq:block_solution_identity_general}
\mathsf{F}(z)\, \mathsf{H}(z) = \mathrm{Id}_{2^{n-1}}.
\end{equation}
\end{prop}

\begin{proof}
The formula \eqref{eq:G_avg} follows from iterating
\[
\mathsf{P}^{(k)} = \tfrac12(I + \mathcal T_k) \mathsf{P}^{(k-1)},
\]
together with the pairwise commutativity of the $\mathcal T_k$'s established in Proposition \ref{prop:Tk_properties}(2): expanding the product
\[
\mathsf{H} = \frac{1}{2^{n-1}}\Bigl(\prod_{k=2}^n (I + \mathcal T_k)\Bigr) \mathsf{P} = \frac{1}{2^{n-1}}\sum_{S \subseteq \{2,\ldots,n\}} \mathcal T_S(\mathsf{P}).
\]
The projection property holds because $\frac12(I + \mathcal T_k)$ is the projection onto the $+1$-eigenspace of $\mathcal T_k$ (since $\mathcal T_k^2 = \mathrm{Id}_{2^{n-1}}$ by Proposition \ref{prop:Tk_properties}(1)), and these projections commute. Finally, by Proposition \ref{prop:Tk_properties}(3), each term $\mathcal T_S(\mathsf{P})$ satisfies $\mathsf{F} \mathcal T_S(\mathsf{P}) = \mathrm{Id}_{2^{n-1}}$, so the average $\mathsf{H}$ does as well.
\end{proof}

\subsection{Recovery of the slice monogenic solution}

The joint fixed subspace of $\{\mathcal T_k\}_{k=2}^n$ inside $H^\infty(\mathbb{D})^{N\cdot 2^{n-1} \times 2^{n-1}}$ is precisely the subspace of matrices with the slice monogenic block structure
\[
\mathsf{H} = \begin{pmatrix} M_{h_1} \\ \vdots \\ M_{h_N} \end{pmatrix},
\]
where each $M_{h_j}$ has the form \eqref{eq:Mf-matrix-section} for some $H_A^{(j)} \in H^\infty(\mathbb{D})$, $A \in \mathcal I$.

Substituting this block form into \eqref{eq:block_solution_identity_general} yields
\[
M_{f_1} M_{h_1} + \cdots + M_{f_N} M_{h_N} = \mathrm{Id}_{2^{n-1}},
\]
which by the homomorphism property of the matrix representation $f_I \mapsto M_{f}$ is equivalent to the Clifford B\'ezout identity
\[
f_{I,1} * h_{I,1} + \cdots + f_{I,N} * h_{I,N} = 1 \qquad \text{on}\ \mathbb B\cap\mathbb{C}_I\cong \mathbb D,
\]
and via extension theorem, see Theorem \ref{thm:extension}, we obtain uniquely determined slice monogenic functions $h_1,\ldots , h_N$ defined on $\mathbb B$ and solving $f_1*h_1+\cdots +f_N*h_N=1$.
We have therefore proved:

\begin{thm}[General Clifford corona theorem]\label{thm:corona-general}
Let $n \ge 2$, and $N\in\mathbb{N}\cup\{\infty\}$ and let $f_{1}, \ldots, f_{N}$ be bounded slice monogenic functions on $\mathbb{B}\subset\mathbb{R}^{n+1}$, $\mathbb{R}_n$-valued. If the block matrix $\mathsf{F} = (M_{f_1} \cdots M_{f_N})$ defined on $\mathbb D$ (isomorphic to $\mathbb{B} \cap \mathbb{C}_I$ for $I\in\mathbb S$) admits a bounded right inverse $\mathsf{P}\in H^{\infty}(\mathbb{D})^{N 2^{n-1}\times 2^{n-1}}$ via the matrix Corona theorem, then there exist bounded slice monogenic functions $h_{1}, \ldots, h_{N}$ on $\mathbb{B}$ satisfying the B\'ezout identity
\[
f_{1} * h_{1} + \cdots + f_{N} * h_{N} = 1,
\]
and $\max_j \|h_j\|_{H^\infty(\mathbb B)}\lesssim \|\mathsf P\|_\infty$.
\end{thm}

This result implies the following, see Theorem \ref{thm:main-intro}:
\begin{thm}[Slice monogenic Corona theorem]
Let $n\ge 2$, $N\in\mathbb{N}\cup\{\infty\}$, and $0<\delta<1$.  Let $\mathbb B\subseteq\mathbb{R}^{n+1}$ be the unit ball and let $f_{1},\ldots,f_{N}\in H^\infty(\mathbb B)$ be slice monogenic functions valued in $\mathbb{R}_n$.  Suppose the data satisfy the Carleson-type condition
\[
0<\delta^2\le \sum_{j=1}^N |f_{j}(q)|^2\le 1\quad\text{for all }q\in \mathbb B,
\]
together with the second, quadratic Carleson condition arising from $\det(\mathsf{F}\,\mathsf{F}^\ast)\ge\delta^2$,  $\forall z\in\mathbb{B}\cap\mathbb{C}_I$ for the block matrix $\mathsf{F}=(M_{f_1}\cdots M_{f_N})$ (made explicit in Theorem~\ref{thm:master-formula}).  Then there exist slice monogenic functions $h_{1},\ldots,h_{N}\in H^\infty(\mathbb B)$ satisfying the B\'ezout identity
\[
\sum_{j=1}^{N} (f_{j}\ast h_{j})(q)=1\quad\text{for all }q\in \mathbb B,
\]
with norm estimates
\begin{equation}
\max_{1\le j\le N}\norm{h_{j}}_{H^\infty(\mathbb B)}\lesssim C(\delta,n),\label{estim}
\end{equation}
where $C(\delta,n)$ is the constant in the matrix-Corona estimate for the $2^{n-1}\times N\cdot 2^{n-1}$ block matrix $\mathsf{F}$.  In particular, $C(\delta,n)$ depends polynomially on $1/\delta$ with exponent and coefficients depending on $n$; for $n=2$ it specializes to the explicit constant $C(\delta)$ of \cite{CPSW}.
\end{thm}

\begin{proof}
Let us fix an arbitrary $I\in\mathbb S$ and let us consider the matrix $\mathsf{F}$ as in the statement. The diagonal elements of $\mathsf{F}\mathsf{F}^*$ are of the form
$$
(\mathsf{F}\mathsf{F}^*)_{A,A}=\sum_{j=1}^N\sum_{B\in\mathcal{I}} |(M_{f_j})_{A,B}(z)|^2=\begin{cases}
\sum_{j=1}^N\sum_{C\in\mathcal I} |F_C^{(j)}(z)|^2=\sum_{j=1}^N |f_i(z)|^2, & |A|\text{ even,}\\
\sum_{j=1}^N\sum_{C\in\mathcal I} |F_C^{(j)}(\bar z)|^2=\sum_{j=1}^N |f_i(\bar z)|^2, & |A|\text{ odd.}
\end{cases}
$$
Thus, by hypothesis, we have $(\mathsf{F}\mathsf{F}^*)_{A,A}\leq 1$ and ${\rm Tr}(\mathsf{F}\mathsf{F}^*) \leq m:=2^{n-1}$.  Apply the spectral theorem and denote the (real) eigenvalues of $\mathsf{F}\mathsf{F}^*$ by $\lambda_1\leq\ldots\leq\lambda_m$ and use the well-known inequality between arithmetic and geometric means which gives $\lambda_2\cdots \lambda_m\leq ( {\rm Tr}(\mathsf{F}\mathsf{F}^*)/(m-1))^{m-1}$. Since $\lambda_1=\det(\mathsf{F}\mathsf{F}^*)/(\lambda_2\cdots \lambda_m)\geq \delta^2 \left((m-1)/m\right)^{m-1}=:\eta$. We deduce that  $\mathsf{F}\mathsf{F}^*\geq\lambda_1 {\rm Id}_m\geq \eta {\rm Id}_m$. Therefore we can apply
Theorem \ref{thm:corona-general} and deduce that there exist functions $h_1,\ldots, h_N$ slice monogenic in $\mathbb B\subset\mathbb{R}^{n+1}$ solutions to the problem. The fact that $h_j$ satisfy the norm estimate \eqref{estim} follows from the representation formula \eqref{eq:representation}.
\end{proof}
\subsection{Specializations}

\begin{itemize}
\item For $n = 2$ (the quaternionic case, $\mathbb{R}_2 \cong \mathbb{H}$), there is a single involution $\mathcal T_2$ and the symmetrized solution is the single-step average $\mathsf{H} = \tfrac12(\mathsf{P} + \mathcal T_2(\mathsf{P}))$, recovering the quaternionic case.

\item For $n = 3$, there are two commuting involutions $\mathcal T_2, \mathcal T_3$ and the symmetrized solution involves the four-term average
\[
\mathsf{H} = \tfrac14\bigl(\mathsf{P} + \mathcal T_2(\mathsf{P}) + \mathcal T_3(\mathsf{P}) + \mathcal T_2 \mathcal T_3(\mathsf{P})\bigr),
\]
recovering the formulation of the previous subsection.

\item For general $n$, the symmetrization is a $2^{n-1}$-term average indexed by subsets $S \subseteq \{2, \ldots, n\}$.
\end{itemize}

\section{Appendix: The Determinant computation via Cauchy--Binet}
\label{s:LinearAlgebra}

In this appendix we collect the relevant linear algebra facts that allow us to see the corresponding Carleson condition on the data.

The idea will be to apply the Cauchy--Binet identity to the matrix  $\mathcal{M}(z)=\mathsf{F}(z) \mathsf{F}(z)^\ast$.  This expands the matrix in terms of minors or non-negative quantities.

\subsection{\texorpdfstring{The Quaternionic Case: $n=2$}{The Quaternionic Case: n=2}}

\begin{thm}\label{thm:detN-CB}
Let $F_1,\ldots,F_N$ and $G_1,\ldots,G_N$ be holomorphic functions on $\mathbb{D}$, and write $\hat{h}(z) := \overline{h(\overline{z})}$ for the Schwarz reflection (so $\hat h$ is holomorphic). Define the $2 \times 2N$ matrix
\begin{equation}\label{eq:tildeF-N}
\mathsf{F}(z) :=
\begin{pmatrix}
\begin{array}{cccccc}
F_1(z) & -G_1(z) & \cdots & F_N(z) & -G_N(z) \\
\hat G_1(z) & \hat F_1(z) & \cdots & \hat G_N(z) & \hat F_N(z)
\end{array}
\end{pmatrix},
\end{equation}
and let $\mathcal{M}(z) := \mathsf{F}(z)\mathsf{F}^*(z)$ be the associated $2 \times 2$ matrix. Then
\begin{equation}\label{eq:detN-CB}
\begin{aligned}
\det\mathcal{M}(z)
&= \sum_{r=1}^{N}\sum_{j=1}^{N} \bigl| F_j(z)\hat F_r(z) + G_r(z)\hat G_j(z)\bigr|^2 \\
&\quad + \sum_{1 \le r < j \le N} \bigl| F_j(z)\hat G_r(z) - F_r(z)\hat G_j(z)\bigr|^2 \\
&\quad + \sum_{1 \le r < j \le N} \bigl| G_j(z)\hat F_r(z) - G_r(z)\hat F_j(z)\bigr|^2.
\end{aligned}
\end{equation}
\end{thm}

\begin{proof}
Throughout the proof we suppress the variable $z$. The Gram matrix is
\[
\mathcal{M} =  \mathsf{F}\mathsf{F}^* =
\begin{pmatrix}
\displaystyle\sum_{j=1}^{N}\bigl(|F_j|^2 + |G_j|^2\bigr) & \displaystyle\sum_{j=1}^{N}\bigl(F_j\,\overline{\hat G_j} - G_j\,\overline{\hat F_j}\bigr)
 \\[0.6em]
\displaystyle\sum_{j=1}^{N}\bigl(\overline{F_j}\,\hat G_j - \overline{G_j}\,\hat F_j\bigr) &
\displaystyle\sum_{j=1}^{N}\bigl(|\hat F_j|^2 + |\hat G_j|^2\bigr)
\end{pmatrix},
\]
which matches the $2 \times 2$ matrix whose determinant we wish to compute (after using $|\hat F_j(z)|^2 = |F_j(\overline z)|^2$ and $\overline{\hat G_j(z)} = G_j(\overline z)$).

\bigskip

The key tool that will be used is the Cauchy--Binet formula, which we recall.  For any $m \times n$ matrix $A$ with $m \le n$,
\begin{equation}\label{eq:CB}
\det(A A^*) = \sum_{\substack{S \subseteq \{1,\ldots,n\} \\ |S| = m}} \bigl|\det A_S\bigr|^2,
\end{equation}
where $A_S$ is the $m \times m$ submatrix obtained by selecting the rows of $A^*$ indexed by $S$, equivalently, the columns of $A$ indexed by $S$ when $A$ is read row-wise. Applied to $\mathsf{F}$ (which is $2 \times 2N$), with $m = 2$, this gives
\begin{equation}\label{eq:CB-applied}
\det\mathcal{M} = \det( \mathsf{F} \mathsf{F}^*) = \sum_{1 \le i < k \le 2N} \bigl|\det F_{(i,k)}\bigr|^2,
\end{equation}
where $ F_{(i,k)}$ denotes the $2 \times 2$ submatrix formed by columns $i$ and $k$ of $F$.

Observe that the total number of such column pairs is $\binom{2N}{2} = 2N^2 - N$.

\bigskip

To facilitate the computation we split the corresponding $2 \times 2$ minors into certain collections. Index the columns of $\mathsf{F}$ as
\begin{align*}
\text{(odd)} \quad 1,\ldots,N: & \ \text{columns of the form } (F_j, \hat G_j)^{\top},\\
\text{(even)} \quad 1,\ldots,N:&  \ \text{columns of the form } (-G_{j}, \hat F_{j})^{\top}.
\end{align*}
A column pair $(i,k)$ falls into one of three cases.

\smallskip

\textit{Case 1: Both columns from the odd collection.} For $1 \le r < j \le N$,
\[
\det \begin{pmatrix} F_r & F_j  \\ \hat G_r & \hat G_j \end{pmatrix}
= F_r \hat G_j - F_j \hat G_r.
\]
Up to sign this is $F_j \hat G_r - F_r \hat G_j$, and the squared modulus is invariant. The contribution from these $\binom{N}{2}$ minors is
\[
\sum_{1 \le r < j \le N} |F_j \hat G_r - F_r \hat G_j|^2.
\]

\smallskip

\textit{Case 2: Both columns from the even collection.} For $1 \le r < j \le N$,
\[
\det \begin{pmatrix} -G_r & -G_j \\ \hat F_r  & \hat F_j \end{pmatrix}
= -G_r \hat F_j + G_j \hat F_r = G_j \hat F_r - G_r \hat F_j.
\]
The contribution from these $\binom{N}{2}$ minors is
\[
\sum_{1 \le r < j \le N} |G_j \hat F_r - G_r \hat F_j|^2.
\]

\smallskip

\textit{Case 3: One column from each collection.} For $r \in \{1,\ldots,N\}$ (odd) and $j \in \{1,\ldots,N\}$ (even),
\[
\det \begin{pmatrix} F_r & -G_j \\ \hat G_r & \hat F_j \end{pmatrix}
= F_r \hat F_j + G_j \hat G_r.
\]
This gives $N^2$ minors (one per ordered pair $(r,j) \in \{1,\ldots,N\}^2$). Their contribution is
\[
\sum_{r=1}^{N}\sum_{j=1}^{N} |F_r \hat F_j + G_j \hat G_r|^2.
\]

\bigskip

This is assembled by substituting the three cases into \eqref{eq:CB-applied},
\[
\det\mathcal{M}
= \sum_{r,j=1}^{N} |F_r \hat F_j + G_j \hat G_r|^2
+ \sum_{1\leq r<j\leq N} |F_j \hat G_r - F_r \hat G_j|^2
+ \sum_{1\leq r<j\leq N} |G_j \hat F_r - G_r \hat F_j|^2,
\]
which agrees with \eqref{eq:detN-CB} after the trivial relabeling $(r,j) \leftrightarrow (j,r)$ in the first sum (which runs over all ordered pairs, so the sum is unchanged).  The minor counts are $\binom{N}{2} + \binom{N}{2} + N^2 = 2N^2 - N = \binom{2N}{2}$, as required.
\end{proof}

\begin{rem}\label{rmk:CB-PSD}
Since $\mathcal{M} = \mathsf{F} \mathsf{F}^*$, the matrix $\mathcal{M}$ is positive semidefinite, and \eqref{eq:detN-CB} exhibits this nonnegativity manifestly as a sum of $\binom{2N}{2}$ squared moduli.
\end{rem}

\subsection{\texorpdfstring{The Clifford Case: $n=3$}{The Clifford Case: n=3}}

In this section we compute $\det( \mathsf{F}\,\mathsf{F}^{\ast})$ for the block matrix
\[
\mathsf{F}(z)=\bigl(M_{f_1}(z)\ \cdots\ M_{f_N}(z)\bigr)\in H^\infty(\mathbb{D})^{4\times 4N}
\]
introduced in \eqref{eq:blockF_clifford}. Just as in the $2 \times 2N$ complex case, Theorem \ref{thm:detN-CB}, the result follows from a direct application of the Cauchy--Binet formula combined with a Laplace expansion of each $4\times 4$ minor along its block-column structure.

\subsubsection{Notation: $2\times 2$ minors of $M_f$}

For each $j \in \{1, \ldots, N\}$, each pair of row indices $R \subseteq \{1,2,3,4\}$ with $|R|=2$, and each pair of column indices $C \subseteq \{1,2,3,4\}$ with $|C|=2$, denote
\begin{equation}\label{eq:Mf-minor-22}
m_R^{C}(f_j) := \det\bigl(M_{f_j}\bigr)_{R, C},
\end{equation}
the $2\times 2$ minor of $M_{f_j}$ with rows indexed by $R$ and columns indexed by $C$. There are $\binom{4}{2}^2 = 36$ such minors for each $j$.

\begin{lemma}\label{lem:22-minors-of-Mf}
Each $m_R^{C}(f_j)$ is a polynomial of degree two in the eight scalar components $\{F_k^{(j)}, \hat F_k^{(j)} : k \in \{1,2,3,23\}\}$, with at most two monomial terms. The complete table of $2\times 2$ minors of $M_{f_j}$ is provided in Section \ref{app:M-minors}.
\end{lemma}

\begin{thm}\label{thm:clifford-CB}
With $\mathsf{F}$ as above,
\begin{equation}\label{eq:clifford-CB}
\det\bigl(\mathsf{F}(z)\,\mathsf{F}^*(z)\bigr)
\;=\;
\sum_{\substack{S \subseteq \{1,\ldots,4N\} \\ |S|=4}}\bigl|\det F_S(z)\bigr|^2,
\end{equation}
where, for each $4$-element column subset $S$, the corresponding $4\times 4$ minor admits the Laplace block-expansion
\begin{equation}\label{eq:laplace-minor}
\det F_S \;=\; \sum_{\substack{R_1, \ldots, R_N \subseteq \{1,2,3,4\} \\ |R_j|=|C_j|,\ \bigsqcup R_j = \{1,2,3,4\}}}\!\!\mathrm{sgn}(R_1, \ldots, R_N)\, \prod_{j=1}^{N} m_{R_j}^{C_j}(f_j),
\end{equation}
in which $C_j := \{k : k\text{-th column of }M_{f_j}\text{ lies in }S\}$ is the column-pattern of $S$ within block $j$ (so $|C_1| + \cdots + |C_N| = 4$), and $\mathrm{sgn}(R_1,\ldots,R_N) \in \{\pm 1\}$ is the sign of the permutation $(R_1, R_2, \ldots, R_N) \mapsto (1, 2, 3, 4)$.
\end{thm}

\begin{proof}
The Cauchy--Binet formula applied to the $4 \times 4N$ matrix $\mathsf{F}$ gives \eqref{eq:clifford-CB} immediately, with the sum running over the $\binom{4N}{4}$ four-element column subsets of $\mathsf{F}$.

For \eqref{eq:laplace-minor}, fix a column subset $S$ and write $|C_j| = k_j$ for the number of selected columns coming from block $j$, so $\displaystyle\sum_j k_j = 4$ and the columns of $ F_S$ split into blocks of widths $k_1, \ldots, k_N$. Expanding $\det F_S$ by the generalized Laplace formula along this block-column partition,
\[
\det F_S = \sum_{R_1 \sqcup \cdots \sqcup R_N = \{1,2,3,4\},\; |R_j|=k_j} \mathrm{sgn}(R_1,\ldots, R_N)\, \prod_{j=1}^N \det\bigl(M_{f_j}\bigr)_{R_j, C_j},
\]
which is \eqref{eq:laplace-minor}.
\end{proof}

\begin{rem}
The classification of $S$ by its block-column pattern $(k_1, \ldots, k_N)$ with $\sum_j k_j = 4$ organizes the sum \eqref{eq:clifford-CB} into five types according to how $4$ is partitioned into nonnegative integers $\le 4$, namely
\[
4 = 4,\quad 3+1,\quad 2+2,\quad 2+1+1,\quad 1+1+1+1.
\]
The counts for each type are summarized in Table \ref{tab:clifford-types}.
\end{rem}

\begin{table}[htbp]
\centering
\begin{tabular}{c|c|c}
Partition of 4 & Number of patterns $(k_1,\ldots,k_N)$ & Minors per pattern\\
\hline
$4$ & $N$ & $1$\\
$3+1$ & $ N(N-1)$ & $\binom{4}{3}\binom{4}{1} = 16$\\
$2+2$ & $\binom{N}{2}$ & $\binom{4}{2}^2 = 36$\\
$2+1+1$ & $\binom{N}{1}\binom{N-1}{2}$ & $\binom{4}{2}\binom{4}{1}^2 = 96$\\
$1+1+1+1$ & $\binom{N}{4}$ & $\binom{4}{1}^4 = 256$
\end{tabular}
\caption{Classification of the $\binom{4N}{4}$ minors in \eqref{eq:clifford-CB} by block-column pattern.}
\label{tab:clifford-types}
\end{table}

\subsubsection{The $2\times 2$ minors of $M_f$}
\label{app:M-minors}

Recall the structured $4\times 4$ matrix
\[
M_f =
\begin{pmatrix}
F_1 & -F_2 & -F_3 & -F_{23} \\
\hat F_2 & \hat F_1 & -\hat F_{23} & \hat F_3 \\
\hat F_3 & \hat F_{23} & \hat F_1 & -\hat F_2 \\
F_{23} & -F_3 & F_2 & F_1
\end{pmatrix},
\]
where for brevity we have dropped the function index $j$ and the variable $z$. For each pair of row indices $R = \{r_1, r_2\}$ and column indices $C = \{c_1, c_2\}$ with $r_1 < r_2$ and $c_1 < c_2$, we denote
\[
m_R^C(f) := \det\bigl(M_f\bigr)_{R, C},
\]
the corresponding $2\times 2$ minor of $M_f$. There are $\binom{4}{2}^2 = 36$ such minors, each a polynomial of degree two in the eight Clifford components $F_1, F_2, F_3, F_{23}, \hat F_1, \hat F_2, \hat F_3, \hat F_{23}$.

Table \ref{tab:Mf-22-minors} below lists all 36 minors, organized by row pair.

\begin{table}[ht]
\centering
\begin{tabular}{c|c|l}
$R$ & $C$ & $m_R^C(f)$ \\
\hline\hline
\multirow{6}{*}{$(1,2)$}
& $(1,2)$ & $F_1 \hat F_1 + F_2 \hat F_2$ \\
& $(1,3)$ & $-F_1 \hat F_{23} + F_3 \hat F_2$ \\
& $(1,4)$ & $F_1 \hat F_3 + F_{23} \hat F_2$ \\
& $(2,3)$ & $F_2 \hat F_{23} + F_3 \hat F_1$ \\
& $(2,4)$ & $-F_2 \hat F_3 + F_{23} \hat F_1$ \\
& $(3,4)$ & $-F_3 \hat F_3 - F_{23} \hat F_{23}$ \\
\hline
\multirow{6}{*}{$(1,3)$}
& $(1,2)$ & $F_1 \hat F_{23} + F_2 \hat F_3$ \\
& $(1,3)$ & $F_1 \hat F_1 + F_3 \hat F_3$ \\
& $(1,4)$ & $-F_1 \hat F_2 + F_{23} \hat F_3$ \\
& $(2,3)$ & $-F_2 \hat F_1 + F_3 \hat F_{23}$ \\
& $(2,4)$ & $F_2 \hat F_2 + F_{23} \hat F_{23}$ \\
& $(3,4)$ & $F_3 \hat F_2 + F_{23} \hat F_1$ \\
\hline
\multirow{6}{*}{$(1,4)$}
& $(1,2)$ & $-F_1 F_3 + F_2 F_{23}$ \\
& $(1,3)$ & $F_1 F_2 + F_3 F_{23}$ \\
& $(1,4)$ & $F_1^2 + F_{23}^2$ \\
& $(2,3)$ & $-F_2^2 - F_3^2$ \\
& $(2,4)$ & $-F_1 F_2 - F_3 F_{23}$ \\
& $(3,4)$ & $-F_1 F_3 + F_2 F_{23}$ \\
\hline
\multirow{6}{*}{$(2,3)$}
& $(1,2)$ & $-\hat F_1 \hat F_3 + \hat F_2 \hat F_{23}$ \\
& $(1,3)$ & $\hat F_1 \hat F_2 + \hat F_3 \hat F_{23}$ \\
& $(1,4)$ & $-\hat F_2^2 - \hat F_3^2$ \\
& $(2,3)$ & $\hat F_1^2 + \hat F_{23}^2$ \\
& $(2,4)$ & $-\hat F_1 \hat F_2 - \hat F_3 \hat F_{23}$ \\
& $(3,4)$ & $-\hat F_1 \hat F_3 + \hat F_2 \hat F_{23}$ \\
\hline
\multirow{6}{*}{$(2,4)$}
& $(1,2)$ & $-F_3 \hat F_2 - F_{23} \hat F_1$ \\
& $(1,3)$ & $F_2 \hat F_2 + F_{23} \hat F_{23}$ \\
& $(1,4)$ & $F_1 \hat F_2 - F_{23} \hat F_3$ \\
& $(2,3)$ & $F_2 \hat F_1 - F_3 \hat F_{23}$ \\
& $(2,4)$ & $F_1 \hat F_1 + F_3 \hat F_3$ \\
& $(3,4)$ & $-F_1 \hat F_{23} - F_2 \hat F_3$ \\
\hline
\multirow{6}{*}{$(3,4)$}
& $(1,2)$ & $-F_3 \hat F_3 - F_{23} \hat F_{23}$ \\
& $(1,3)$ & $F_2 \hat F_3 - F_{23} \hat F_1$ \\
& $(1,4)$ & $F_1 \hat F_3 + F_{23} \hat F_2$ \\
& $(2,3)$ & $F_2 \hat F_{23} + F_3 \hat F_1$ \\
& $(2,4)$ & $F_1 \hat F_{23} - F_3 \hat F_2$ \\
& $(3,4)$ & $F_1 \hat F_1 + F_2 \hat F_2$ \\
\end{tabular}

\caption{The 36 $2\times 2$ minors of $M_f$. Each entry is the determinant of the $2\times 2$ submatrix with the specified rows and columns.}
\label{tab:Mf-22-minors}
\end{table}

\begin{rem}
The minors $m_R^C(f)$ exhibit several structural features:
\begin{enumerate}
\item \emph{Pure $\mathsf{F}$-blocks:} Minors with row pair $R = (1,4)$ involve only the unhatted components $F_1, F_2, F_3, F_{23}$. These come from rows $1, 4$ of $M_f$, which contain only unhatted entries.

\item \emph{Pure $\hat{\mathsf{F}}$-blocks:} Minors with row pair $R = (2,3)$ involve only the hatted components $\hat F_1, \hat F_2, \hat F_3, \hat F_{23}$, from rows $2, 3$.

\item \emph{Mixed blocks:} Minors with row pairs $R \in \{(1,2), (1,3), (2,4), (3,4)\}$ mix both unhatted and hatted components.

\item \emph{Coincidences:} Certain pairs of minors coincide or differ only by sign --- for example, $m_{(1,2)}^{(1,2)} = m_{(3,4)}^{(3,4)} = F_1\hat F_1 + F_2\hat F_2$, and $m_{(1,4)}^{(1,2)} = m_{(1,4)}^{(3,4)} = -F_1 F_3 + F_2 F_{23}$. These coincidences reflect the Clifford symmetries \eqref{eq:symmetry_Mf}.
\end{enumerate}
\end{rem}

\subsection{\texorpdfstring{The Clifford Case: General $n$}{The Clifford Case: General n}}

We now state and prove the general $\mathbb{R}_n$ analogue of the determinant identities of the previous subsections. The result follows from the Cauchy--Binet formula applied to the block matrix $\mathsf{F}$, combined with a generalized Laplace expansion of each minor along the block-column structure.

\subsubsection{Setup}

Let $n \ge 2$ and let $f_{1}, \ldots, f_{N}$ be slice monogenic functions $\mathbb{R}_n$-valued. For each $j \in \{1,\ldots,N\}$, write their restrictions to $\mathbb{C}_I$ as
\[
f_{I,j}(z) = \sum_{A \in \mathcal I} F_A^{(j)}(z)\, I_A,
\]
and let $M_{f_j}$ denote the associated $2^{n-1} \times 2^{n-1}$ matrix defined by \eqref{eq:Mf-matrix-section}. Form the block matrix
\begin{equation}\label{eq:blockF-general}
\mathsf{F}(z) := \bigl(M_{f_1}(z)\ \cdots\ M_{f_N}(z)\bigr) \in H^\infty(\mathbb{D})^{2^{n-1}\times N \cdot 2^{n-1}}.
\end{equation}

For each $j \in \{1,\ldots,N\}$ and each pair of subsets $R, C \subseteq \mathcal I$ with $|R| = |C| = r$, denote by
\begin{equation}\label{eq:Mf-minor-RC}
m_R^C(f_j) := \det\bigl[(M_{f_j})_{A,B}\bigr]_{A \in R,\, B \in C}
\end{equation}
the $r\times r$ minor of $M_{f_j}$ with rows indexed by $R$ and columns indexed by $C$. By the formula \eqref{eq:Mf-matrix-section}, $m_R^C(f_j)$ is a polynomial of degree $r$ in the components $\{F_A^{(j)}, \hat F_A^{(j)} : A \in \mathcal I\}$ with explicit signs determined by $\sigma$.

\subsubsection{Statement of the Master Formula}

\begin{thm}[Master determinant formula]\label{thm:master-formula}
Fix any total order on $\mathcal I$, with respect to which signs of permutations are computed. Then
\begin{equation}\label{eq:master}
\det\bigl(\mathsf{F}(z)\,\mathsf{F}^{\ast}(z)\bigr)
\;=\; \sum_{S} \,\Biggl|\,\sum_{(R_1,\ldots,R_N)} \mathrm{sgn}(R_1, \ldots, R_N) \prod_{j=1}^{N} m_{R_j}^{C_j}(f_j)\,\Biggr|^2,
\end{equation}
where:
\begin{itemize}
\item the outer sum runs over column subsets $S \subseteq \{1, \ldots, N \cdot 2^{n-1}\}$ with $|S| = 2^{n-1}$, decomposed as a disjoint union $S = C_1 \sqcup \cdots \sqcup C_N$ in which $C_j \subseteq \mathcal I$ records the columns of $S$ lying in block $j$ (so $\sum_j |C_j| = 2^{n-1}$);
\item the inner sum runs over ordered row-partitions $(R_1, \ldots, R_N)$ of $\mathcal I$ satisfying $|R_j| = |C_j|$ for each $j$;
\item $\mathrm{sgn}(R_1, \ldots, R_N) \in \{\pm 1\}$ is the sign of the permutation that takes the concatenation $(R_1, R_2, \ldots, R_N)$ (with each $R_j$ listed in increasing order) to the natural order of $\mathcal I$.
\end{itemize}
The total number of squared-modulus terms in \eqref{eq:master} is $\displaystyle\binom{N \cdot 2^{n-1}}{2^{n-1}}$, and the matrix $ \mathsf{F}\, \mathsf{F}^*$ is positive semidefinite.
\end{thm}

\begin{proof}
We apply Cauchy--Binet to $\mathsf{F}$. Since $\mathsf{F}$ is $2^{n-1} \times N\cdot 2^{n-1}$ with row dimension at most column dimension, the Cauchy--Binet formula gives
\begin{equation}\label{eq:CB-master}
\det(\mathsf{F}\,\mathsf{F}^{\ast}) = \sum_{\substack{S \subseteq \{1,\ldots,N\cdot 2^{n-1}\}\\ |S| = 2^{n-1}}} |\det F_S|^2,
\end{equation}
where $F_S$ is the $2^{n-1}\times 2^{n-1}$ submatrix of $\mathsf{F}$ consisting of the columns indexed by $S$. The number of terms in this sum is $\binom{N\cdot 2^{n-1}}{2^{n-1}}$.

Fix one such $S$ and write $S = C_1 \sqcup \cdots \sqcup C_N$ as in the theorem. The matrix $F_S$ is a $2^{n-1}\times 2^{n-1}$ matrix whose columns split into $N$ blocks of widths $|C_1|, \ldots, |C_N|$, with the $j$-th block consisting of the columns $C_j$ of $M_{f_j}$. The \emph{generalized Laplace expansion} of $\det F_S$ along this block-column partition states
\begin{equation}\label{eq:laplace-general}
\det F_S = \sum_{\substack{(R_1, \ldots, R_N)\\ |R_j| = |C_j|,\ \bigsqcup_j R_j = \mathcal I}} \mathrm{sgn}(R_1, \ldots, R_N) \prod_{j=1}^{N} \det\bigl[(M_{f_j})_{A,B}\bigr]_{A \in R_j,\, B \in C_j}.
\end{equation}
The product on the right is precisely $\prod_j m_{R_j}^{C_j}(f_j)$ in the notation \eqref{eq:Mf-minor-RC}. Substituting \eqref{eq:laplace-general} into \eqref{eq:CB-master} and taking absolute values yields \eqref{eq:master}.
\end{proof}

\subsubsection{Comments on the structure}

The Master Formula \eqref{eq:master} is the natural unification of the special cases established in the previous subsections:
\begin{itemize}
\item For $n=2$ (the quaternionic case), $|\mathcal I| = 2$, each $F_S$ is a $2\times 2$ matrix, and the inner Laplace sum has at most $\binom{2}{|C_1|}$ terms per row-partition, recovering Theorem \ref{thm:detN-CB}.
\item For $n=3$ (the Clifford $\mathbb{R}_3$ case), $|\mathcal I| = 4$, each $F_S$ is $4\times 4$, and the inner Laplace sum is the one displayed in Theorem \ref{thm:clifford-CB}.
\end{itemize}

Two practical features of the formula are worth noting. First, the $r \times r$ minors $m_R^C(f_j)$ appearing in the inner products are determined entirely by the matrix entries \eqref{eq:Mf-matrix-section}; in particular, each $m_R^C(f_j)$ is a homogeneous polynomial of degree $r$ in the variables $\{F_A^{(j)}, \hat F_A^{(j)} : A \in \mathcal I\}$. Second, although the number of squared-modulus terms grows as $\binom{N \cdot 2^{n-1}}{2^{n-1}}$, the individual terms remain structurally explicit: each is a polynomial of degree $2^{n-1}$ obtained by combining at most $N!$ row-partition contributions, each of which is itself a product of explicit minors. The Master Formula therefore expresses $\det(\mathsf{F}\,\mathsf{F}^{\ast})$ as an explicit, manifestly nonnegative polynomial in the slice components of the data, of uniform structure across all dimensions $n$.

\subsection{Principal minors and Sylvester's criterion}

The Master Formula \eqref{eq:master} computes $\det(\mathsf{F} \mathsf{F}^{\ast})$ explicitly. The same argument applies to any principal submatrix of $\mathsf{F} \mathsf{F}^*$, yielding analogous formulas. We record the result here, both for its own sake and as a determinant-based proof of positive semidefiniteness via Sylvester's criterion.

\begin{thm}[Principal minors]\label{thm:principal-minors}
Fix a total order on $\mathcal I$ and identify the rows of $\mathsf{F}$ with $\mathcal I$ accordingly. For any nonempty subset $T \subseteq \mathcal I$ with $|T| = k$, the principal submatrix $(\mathsf{F} \mathsf{F}^*)_{T,T}$ is a Gram matrix, and
\begin{equation}\label{eq:principal-minor-formula}
\det\bigl( \mathsf{F} \mathsf{F}^* \bigr)_{T,T} = \sum_{\substack{S \subseteq \{1,\ldots,N \cdot 2^{n-1}\}\\ |S| = k}} \bigl|\det F_{T, S}\bigr|^2,
\end{equation}
where $F_{T, S}$ denotes the $k\times k$ submatrix of $F$ with rows indexed by $T$ and columns indexed by $S$. Each minor $\det  F_{T, S}$ admits the Laplace block-expansion
\begin{equation}\label{eq:principal-laplace}
\det F_{T, S} = \sum_{(R_1,\ldots,R_N)} \mathrm{sgn}(R_1,\ldots,R_N) \prod_{j=1}^N m_{R_j, T}^{C_j}(f_j),
\end{equation}
where:
\begin{itemize}
\item $S = C_1 \sqcup \cdots \sqcup C_N$ with $C_j \subseteq \mathcal I$ recording the columns of $S$ in block $j$;
\item $(R_1, \ldots, R_N)$ ranges over ordered partitions of $T$ with $|R_j| = |C_j|$;
\item $m_{R, T}^{C}(f_j) := \det\bigl[(M_{f_j})_{A,B}\bigr]_{A \in R, B \in C}$ is the $|R|\times|C|$ minor of $M_{f_j}$ restricted to rows in $T$ and columns in $C$;
\item $\mathrm{sgn}(R_1, \ldots, R_N)$ is the sign of the permutation taking $(R_1, \ldots, R_N)$ to the natural order of $T$.
\end{itemize}
\end{thm}

\begin{proof}
The submatrix $(\mathsf{F} \mathsf{F}^* )_{T,T}$ has $(A, B)$-entry $\displaystyle \sum_{\ell} (F)_{A, \ell} \overline{(F)_{B, \ell}}$ for $A, B \in T$, which is precisely the $(A, B)$-entry of $F_{T, :} \, F_{T, :}^*$, where $F_{T, :}$ is the $k \times N \cdot 2^{n-1}$ submatrix with rows $T$. Hence
\[
\bigl( \mathsf{F} \mathsf{F}^* \bigr)_{T,T} = F_{T,:}\,F_{T,:}^*  .
\]
Applying Cauchy--Binet to this $k \times N\cdot 2^{n-1}$ matrix yields \eqref{eq:principal-minor-formula}. The Laplace block-expansion \eqref{eq:principal-laplace} of each minor $\det F_{T, S}$ proceeds exactly as in the proof of the Master Formula (Theorem \ref{thm:master-formula}), now with rows restricted to $T$.
\end{proof}

\begin{cor}[Positive semidefiniteness via Sylvester's criterion]\label{cor:psd-sylvester}
The matrix $\mathsf{F} \mathsf{F}^*$ is Hermitian and positive semidefinite.
\end{cor}

\begin{proof}
Hermitianness is immediate from $(AB^*)^* = BA^*$ applied to $A = B = \mathsf{F}$. For positive semidefiniteness, Sylvester's criterion  states that a Hermitian matrix is positive semi-definite if and only if all of its principal minors are nonnegative. By Theorem \ref{thm:principal-minors}, every principal minor $\det(\mathsf{F} \mathsf{F}^*)_{T,T}$ is a sum of squared moduli, hence nonnegative.
\end{proof}

\begin{rem}
The leading principal minors correspond to $T = \{1, 2, \ldots, k\}$ in some fixed ordering of $\mathcal I$. For positive \emph{definiteness}, it suffices to check that all leading principal minors are strictly positive, which by \eqref{eq:principal-minor-formula} is equivalent to the existence of at least one nonvanishing $k\times k$ minor of $F_{[k], :}$ for each $k$. 
\end{rem}

\noindent {\bf Declaration of generative AI and AI-assisted technologies in the manuscript preparation process}. 
  During the preparation of this work the authors used ChatGPT in order to check some computations. After using this tool/service, the authors reviewed and edited the content as needed and take full responsibility for the content of the published article.
  The content of this paper reflects the author’s own work and ideas.

\end{document}